\documentclass[a4paper,11pt]{article}%

\newcounter{dummy}

\usepackage{dina42e}
\usepackage{titlesec}
\usepackage{amsmath, amsthm, amsfonts, amssymb, mathrsfs}%
\allowdisplaybreaks
\usepackage{setspace}
\usepackage{lmodern}
\usepackage[color links = true, linkcolor=black, citecolor=black]{hyperref}
\usepackage{booktabs}
\usepackage{enumitem} 
\usepackage{scalerel}
\usepackage[usestackEOL]{stackengine}
\usepackage{apptools}

\def\dashint{\let\mathchoice\oldmathchoice\,\ThisStyle{\ensurestackMath{%
			\stackinset{c}{.2\LMpt}{c}{.5\LMpt}{\SavedStyle-}{%
				\SavedStyle\phantom{\int}}}%
		\setbox0=\hbox{$\SavedStyle\int\,$}\kern-\wd0}\int%
	\let\mathchoice\newmathchoice}

\newcommand{\R}{\mathbb{R}}
\newcommand{\N}{\mathbb{N}}

\newcommand{\e}{\varepsilon}

\newcommand{\Rtimes}{\mathbb{R}^{3\times 3}}
\newcommand{\Rntimesn}{\mathbb{R}^{n\times n}}
\newcommand{\Colonh}{:_h}
\newcommand{\Sym}{\operatorname{sym}}
\newcommand{\Skew}{\operatorname{skew}}

\newcommand{\divh}{\operatorname{div}_h}

\newcommand{\tr}[1]{\operatorname{tr}_{#1}}

\newcommand{\dist}{\operatorname{dist}}

\newcommand{\B}{\mathcal{B}}
\newcommand{\sym}{\operatorname{sym}}

\makeatletter
\newcommand{\myitem}[1][]{\item[#1]\refstepcounter{dummy}\def\@currentlabel{#1}}
\makeatother

\newtheorem{theorem}{Theorem}[section]
\newtheorem{lem}[theorem]{Lemma}
\newtheorem{Coro}[theorem]{Corollary}

\theoremstyle{definition}

\newtheorem{bem}[theorem]{Remark}

\numberwithin{equation}{section}

\AtAppendix{}

\ifx\cs\documentclass \footheight 12pt \fi \footskip 30pt 
\titleformat{\section}{\normalfont\Large\bfseries}{\thesection}{1em}{}
\titleformat{\subsection}{\normalfont\large\bfseries}{\thesubsection}{1em}{}
\titleformat{\subsubsection}{\normalfont\normalsize\bfseries}{\thesubsubsection}{1em}{}
\titleformat{\paragraph}[runin]{\normalfont\normalsize\bfseries}{\theparagraph}{1em}{}
\titleformat{\subparagraph}[runin]{\normalfont\normalsize\bfseries}{\thesubparagraph}{1em}{}

\begin{document}
\begin{titlepage}
	\title{Convergence of Thin Vibrating Rods to a Linear Beam Equation}
	\author{Helmut Abels\thanks{Fakult\"{a}t f\"{u}r Mathematik, Universit\"{a}t Regensburg, 93040 Regensburg, Germany, e-mail:
			helmut.abels@mathematik.uni-regensburg.de, corresponding author} \space and Tobias Ameismeier\thanks{Fakult\"{a}t f\"{u}r Mathematik, Universit\"{a}t Regensburg, 93040 Regensburg, Germany, e-mail:
			tobias.ameismeier@mathematik.uni-regensburg.de}}
\end{titlepage}
\maketitle
\begin{abstract}
	We show that solutions for a specifically scaled nonlinear wave equation of nonlinear elasticity converge to solutions of a linear Euler-Bernoulli beam system. We construct an approximation of the solution, using a suitable asymptotic expansion ansatz based upon solutions to the one-dimensional beam equation. Following this, we derive the existence of appropriately scaled initial data and can bound the difference between the analytical solution and the approximating sequence.
\end{abstract}
\noindent{\textbf{Key words:}} Wave equation, nonlinear elasticity, thin rods, dimension reduction

\vspace{2mm}	

\noindent{\textbf{AMS-Classification:}} Primary: 74B20, 
Secondary:
35L20, 
35L70,  
74K10 
\section{Introduction}
The relation between solutions of the equations of nonlinear elasticity and solutions for lower dimensional models is of great interest since the lower dimensional models are often easier to analyse and to use for numerical simulations. A general introduction to this topic can be found in \cite{Antman} or for continuum mechanics see \cite{Gurtin}. In dependence of the size of the deformation and the applied forces different lower dimensional models can occur. Therefore a rigoros derivation of the lower dimensional models are of great interest. In the case of the time independent case there are many results, cf.\ e.g.\ Friesecke, James and M\"uller~\cite{FJM2002, FJM2006} for the case of plates and Mora and M\"uller~\cite{MoraMueller2003, MoraMueller2004} in the case of rods and Scardia~\cite{Scardia2006, Scardia2009} for curved rods. But there are only few results in the time independent case so far. 

In this contribution we investigate the relation between solutions of an appropriately scaled wave equation of nonlinear elasticity and solutions of a linear Euler-Bernoulli beam system. More precisely, let $\Omega := [0,L]\times S$ be the reference configuration of a three dimensional rod, where $L>0$ and $S\subset\R^2$ is the cross section. Then we consider the following nonlinear system
\begin{align}\label{eq:1}
	\partial_t^2 u_h - \frac{1}{h^2}\operatorname{div}_h \Big(D\tilde{W}(\nabla_h u_h)\Big) &= h^{2} f_h \quad \text{in } \Omega\times [0, T),\\\label{eq:2}
	D\tilde{W}(\nabla_h u_h)\nu|_{(0, L)\times \partial S} &= 0,\\
	u_h \text{ is $L$-periodic} &\text{ w.r.t. } x_1,\\\label{eq:4}
	(u_h, \partial_t u_h)|_{t=0} &= (u_{0,h}, u_{1,h}),
\end{align}
where $T > 0$ and $\tilde{W}$ is some elastic energy density chosen later. Existence of strong solutions for large times of this system was shown in \cite{AbelsAmeismeier1}. More details on how to justify the scaling can be found there as well. 
The limit system as $h\to 0$ is given by
\begin{gather}\label{eq:limit1}
	\partial_t^2 v + \begin{pmatrix}	I_2 & 0 \\ 0 & I_3 \end{pmatrix} \partial^4_{x_1} v = g\quad\text{in } [0,L]\times [0,\infty)\\
	v \text{ is $L$-periodic in } x_1\\\label{eq:limit3}
	(v,\partial_t v)|_{t=0} = (\tilde{v}_0, \tilde{v}_1)
\end{gather}
where $\tilde{v}_0, \tilde{v}_1$ and $I_2, I_3$ are appropriately chosen initial values and weights, respectively. It is the goal of this manuscript to prove convergence of the solutions of the system \eqref{eq:1}-\eqref{eq:4} to solutions of the limit system \eqref{eq:limit1}-\eqref{eq:limit3} with appropriate convergence rates for well-prepared initial data, cf.\ Theorem~\ref{thm:Main} below.

In an energetic setting the relations between higher dimensional models and lower dimensional ones, using the notion of $\Gamma$-convergence a fundamental contribution was given in \cite{FJM2002}. There the classical geometric rigidity is proven. Using this result it was possible to prove a lot of convergence results in different geometrical situations and scaling regimes in the static setting, see for instance \cite{FJM2006, MoraMueller2004, MoraMueller2003}. In the dynamical case for plates a convergence result can be found in \cite{AbelsMoraMuellerConvergence} and in \cite{FK2020} in the case of viscelasticity. The large times existence and a first order asymptotic for plates was shown in \cite{AbelsMoraMueller}.

In the following we want to explain the main novelties and difficulties of this contribution. In a first step we construct an approximation using the solution of the lower dimensional system. This approximation is constructed such that it solves the linearisation around zero of the nonlinear, three dimensional equation up to an error of order $h^3$. This is done explicitly by determining suitable prefactor functions as solutions of systems on $S$.
Thereafter the main difficulty of this work is to establish existence of suitable initial data in order to ensure large times existence for the solution of the nonlinear problem. This is done in Section \ref{subsection::ExistenceOfAndBoundsOnInitialValues}. Here we use the nonlinear equations for the initial data from the compatability conditions. These are solved via a fixed point argument on precisely chosen function spaces. Finally in Section \ref{subsection::MainResult} the convergence properties are proven. For this we use a general result for solutions of the linearised equation. Moreover we have to carefully treat the rotational parts of the initial data, as the spaces for the fixed point argument do not cover them. For this we use a decomposition and the fact that the elastic energy density is chosen as $\tilde{W}(F) = \operatorname{dist}(Id+F; SO(3))$.\\[1ex]
        The results are part of the second author's PhD thesis~\cite{AmeismeierDiss}.\\[1ex]
\noindent{\textbf{Acknowledgements:} Tobias Ameismeier was supported by the RTG 2339 “Interfaces, Complex Structures,
	and Singular Limits” of the German Science Foundation (DFG). The support is gratefully acknowledged.}

\section{Preliminaries and Auxiliary Results}
\subsection{Notation}
We use  standard notation; in particular $\N$ and $\N_0 := \N \cup\{0\}$ denote the natural numbers with and without zero, respectively. Moreover, the norm on $\R$ and absolute value in $\R^n$, $\Rntimesn$ is denoted by $|.|$ for all $n\in\N$. For $p,~k\in \N$, we denote the classical Lebesgue and Sobolev spaces for some bounded, open set $M\subset\R^n$, by $L^p(M)$, $W^k_p(M)$ and $H^k(M) := W^k_2(M)$. A subscript $(0)$ on a function space will always indicate that elements have zero mean value, e.g., for $g\in H^1_{(0)}(M)$ we have 
\begin{equation}
\int_M g(x) dx = 0.
\end{equation}

The cross section of the rod is always denoted by $S\subset\R^2$ and is assumed to be a smooth and bounded domain. Furthermore be $\Omega_h := (0,L) \times hS \subset \R^3$ for $h \in (0,1]$ and $L>0$ and for convenience we write $\Omega := \Omega_1$. We assume that $S$ satisfies
\begin{align}
\int_S x_2 x_3 dx' &= 0
\label{globalEqualitySecondMomentZero}
\qquad \text{and} \\
\int_S x_2 dx' &= \int_S x_3 dx' = 0,
\label{globalEqualityFirstMomentsZero}
\end{align}
where $x' := (x_2, x_3)\subset \R^2$. This is no loss of generality, as it can always be achieved via a translation and rotation. The scaling shell be such that we can assume $|S|=1$. Furthermore, we denote with $\nabla_h$ the scaled gradient defined as
\begin{equation*}
\nabla_h = \bigg(\partial_{x_1}, \frac{1}{h}\partial_{x_2}, \frac{1}{h}\partial_{x_3}\bigg)^T\quad \text{and}\quad \varepsilon_h(u)= \sym(\nabla_h u).
\end{equation*}
The respective gradient in only $x':=(x_2, x_3)$ direction is denoted by
\begin{equation*}
\nabla_{x'} := \big(\partial_{x_2}, \partial_{x_3}\big)^T.
\end{equation*}
The standard notation $H^k(\Omega)$ and $H^k(\Omega; X)$ is used for $L^2$-Sobolev spaces of order $k\in\N$ with values in $\R$ and some space $X$, respectively.

The space of all $n$-linear mappings $G\colon V^n \to\R$ for a vector space $V$ is denoted, throughout the paper by $\mathcal{L}^n(V)$, $n\in\N$. We deploy the standard identification of $\mathcal{L}^1(\Rntimesn) = (\Rntimesn)'$ with $\Rntimesn$, i.e., $G\in \mathcal{L}^1(\Rntimesn)$ is identified with $A\in\Rntimesn$ such that
\begin{equation*}
G(X) = A : X \quad\text{ for all } X\in\Rntimesn
\end{equation*}
where $A:X = \sum_{i,j=1}^n a_{ij} x_{ij}$ is the usual inner product on $\Rntimesn$. Analogously, for $G\in \mathcal{L}^2(\Rntimesn)$ we use the identification with $\tilde{G}\colon\Rntimesn\to\Rntimesn$ defined by
\begin{equation}
\tilde{G}X : Y = G(X,Y)\quad \text{ for all } X, Y\in\Rntimesn.
\label{IdentificationL2withTensor4Order}
\end{equation}  

We introduce a scaled inner product on $\Rntimesn$
\begin{equation*}
A\Colonh B := \frac{1}{h^2} \Sym A : \Sym B + \Skew A : \Skew B
\end{equation*}
for all $A$, $B\in\Rntimesn$ and $h>0$ and the corresponding norm is denoted by $|A|_h := \sqrt{A\Colonh A}$. With this we can define for $W\in \mathcal{L}^d(\Rntimesn)$ the induced scaled norm by 
\begin{equation*}
|W|_h := \sup_{|A_j|_h \leq 1, j =\{1,\ldots, d\}} |W(A_1, \ldots, A_d)|
\end{equation*}
Using $|A|_h \geq |A|_1 = |A|$ for all $A\in\Rntimesn$ it follows $|W|_h \leq |W|_1 =: |W|$ for all $W\in \mathcal{L}^d(\Rntimesn)$ and $0 < h \leq 1$. The scaled $L^p$-spaces are defined as follows
\begin{equation*}
\|W\|_{L^p_h(U, \mathcal{L}^d(\Rntimesn))} = \|W\|_{L^p_h(U)} = \left( \int_U |W(x)|_h^p dx\right)^{\frac{1}{p}}
\end{equation*}
if $p\in [1, \infty)$, where $U\subset\R^d$ is measurable. Thus $\|W\|_{L^p_h(U; \mathcal{L}^d(\Rntimesn))} \leq \|W\|_{L^p(U;\mathcal{L}^d(\Rntimesn))}$. The scaled norm for $f\in L^p(U, \Rntimesn)$ is defined in the same way
\begin{equation*}
\|f\|_{L^p_h(U,\Rntimesn)} = \|f\|_{L^p_h(U)} = \left( \int_U |f(x)|_h^p dx\right)^{\frac{1}{p}}.
\end{equation*}
Then
\begin{equation*}
\|f\|_{L^p_h(U; \Rntimesn)} \geq \|f\|_{L^p(U;\Rntimesn)}.
\end{equation*}

As we will work with periodic boundary condition in $x_1$-direction we introduce for $m\in\N$
\begin{align*}
H^m_{per}(\Omega) := \Big\{f\in H^m(\Omega) \;:\; \partial^\alpha_x f|_{x_1 = 0} = \partial^\alpha_x f|_{x_1 = L},\;\text{for all } |\alpha| \leq m-1 \Big\}.
\end{align*}
This space can equivalently defined in the following way, which is in some situations more convenient
\begin{equation*}
\tilde{H}^m_{per}(\Omega) := \Big\{f\in H^m_{loc}(\R\times\bar{S}) \;:\; f(x_1, x')= f(x_1 + L, x') \text{ almost everywhere}\Big\}
\end{equation*}
We equipped $\tilde{H}^m_{per}(\Omega)$ with the standard $H^m(\Omega)$-norm. As the maps $f\mapsto f|_{\Omega}\colon \tilde{H}^m_{per}(\Omega) \to H^m_{per}(\Omega)$ and $f\mapsto f_{per}\colon H^m_{per}(\Omega) \to \tilde{H}^m_{per}(\Omega)$ are isomorphisms, we identify $\tilde{H}^m_{per}(\Omega)$ with $H^m_{per}(\Omega)$. This leads immediately to the density of smooth functions in $H^m_{per}(\Omega)$, because, as $S$ is smooth, there exists an appropriate extension operator and thus we can use a convolution argument.

In various estimates we will use an anisotropic variant of $H^k(\Omega)$, as we will have more regularity in lateral direction. Therefore we define
\begin{align*}
H^{m_1, m_2} &(\Omega) := \Big\{u\in L^2(\Omega) \;:\; \partial_{x_1}^l \nabla_{x}^k u\in L^2(\Omega), \text{ for } k=0,\ldots, m_1, l=0,\ldots,m_2\\
&\partial_{x_1}^q \partial_{x}^\alpha u\Big|_{x_1 = 0} = \partial_{x_1}^q\partial_{x}^\alpha u\Big|_{x_1 = L}, \text{ for } q = 0,\ldots, m_1, |\alpha|\leq m_2 \text{ and } q + |\alpha| \leq m_1 + m_2 -1 \Big\}
\end{align*}
where $m_1,~m_2\in\N_0$, the inner product is given by
\begin{equation*}
(f,g)_{H^{m_1, m_2}(\Omega)} = \sum_{k=0,\ldots,m_1; l=0,\ldots m_2} \Big(\partial_{x_1}^l \nabla_x^k f, \partial_{x_1}^l \nabla_x^k g \Big)_{L^2(\Omega)}.
\end{equation*}
Furthermore we will use the scaled norms
\begin{align*}
	\|A\|_{H^m_h(\Omega)} &:= \left(\sum_{|\alpha| \leq m} \|\partial^\alpha_x A\|_{L^2_h(\Omega)}^2\right)^{\frac{1}{2}}\\
	\|B\|_{H^{m_1, m_2}_h(\Omega)} &:= \left(\sum_{k=0,\ldots,m_1; l=0,\ldots,m_2} \|\partial_{x_1}^l \nabla_x^k B\|^2_{L^2_h(\Omega)}\right)^{\frac{1}{2}}.
\end{align*}
for $A\in H^m(\Omega;\R^{n\times n})$ and $B\in H^{m_1,m_2}(\Omega;\R^{n\times n})$ and $n\in\N$. As an abbreviation we denote for $u\in H^k(\Omega;\R^3)$ the symmetric scaled gradient by $\e_h(u):=\Sym(\nabla_h u)$ and $\e(u) = \e_1(u) = \Sym(\nabla u)$.

The following lemma provides the possibility to take traces for $u\in H^{0,1}(\Omega)$:
\begin{lem}
	The operator $\tr{a}\colon H^{0,1}(\Omega) \to L^2(S)$, $u\mapsto u|_{x_1=a}$ is well defined and bounded.
\end{lem}
\begin{proof}
	This is an immediate consequence of the embedding
	\begin{equation*}
	H^{0,1}(\Omega) = H^1(0,L;L^2(S)) \hookrightarrow BUC([0,L];L^2(S))
	\end{equation*}
	where $BUC([0,L];X)$ is the space of all uniformly continuous functions $f\colon [0,L] \to X$ for some Banach space $X$. 
\end{proof}

\subsection{The Strain Energy Density $W$ and Korn's Inequality}
\label{section:StrainEnergyDensityW}
We investigate the mathematical assumptions and resulting properties of the strain-energy density $W$ we use in this contribution. We assume to have $W\colon \Rtimes\to [0,\infty)$ defined by
\begin{equation*}
	W(F) := \frac{1}{2} \dist(F, SO(3))
\end{equation*}
where $SO(3)$ denotes the group of special orthogonal matrices. This energy density clearly satisfies the following general assumptions
\begin{enumerate}
	\item[(i)]{$W\in C^\infty(B_\delta(Id); [0,\infty))$ for some $\delta >0$;}
	\item[(ii)] {$W$ is frame-invariant, i.e. $W(RF) = W(F)$ for all $F\in\Rtimes$ and $R\in SO(3)$;}
	\item[(iii)] {there exists $c_0 > 0$ such that $W(F) \geq c_0 \dist(F, SO(3))^2$ for all $F\in\Rtimes$ and $W(R) = 0$ for every $R\in SO(3)$.}
\end{enumerate}
\begin{bem} \label{RemarkdPropertiesOfElasticEnergyDensity}
	We note that $W$ has a minimum point at the identity, as $W(Id) = 0$ and $W(F)\geq 0$ for all $F\in\Rtimes$. Hence, we have for $\tilde{W}(F):=W(Id + F)$ for all $F\in\Rtimes$, $D\tilde{W}(0)[G] = 0$ for all $G\in\Rtimes$. Moreover, it holds $D^2 \tilde{W}(0) F = \operatorname{sym}F$ and for $P\in\Rtimes_{skew}$, $A, B\in\Rtimes$ we obtain
	\begin{equation}
		D^3 \tilde{W}(0)[A,B,P] = \Big((A^T- A)^T \operatorname{sym}(B) + (B^T - B)^T \operatorname{sym}(A)\Big) : P.
		\label{Global_ThirdOrderD3WEquality}
	\end{equation}
\end{bem}

The following lemma provides an essential decomposition of $D^3\tilde{W}$ in the general form.
\begin{lem}
	There is some constant $C>0$, $\e >0$ and $A\in C^\infty(\overline{B_\e(0)}; \mathcal{L}^3(\Rntimesn))$ such that for all $G\in\Rntimesn$ with $|G| \leq\e$ we have
	\begin{equation*}
	D^3\tilde{W}(G) = D^3\tilde{W}(0) + A(G)
	\end{equation*}
	where
	\begin{alignat}{2}
	|D^3\tilde{W}(0)|_h &\leq Ch&\quad & \text{for all } 0<h\leq 1, \label{Lemma2.6_1}\\
	|A(G)| &\leq C|G|&& \text{for all } |G|\leq\e.\label{Lemma2.6_2}
	\end{alignat}
	\label{Lemma_DecompD3TildeW}
\end{lem}
\begin{proof}
	For the proof we refer to \cite[Lemma 2.6]{AbelsMoraMueller}.
\end{proof}
With this we can prove the following bound for $D^3\tilde{W}$.
\begin{Coro}
	There exist $C$, $\e > 0$ such that
	\begin{equation}
	\|D^3\tilde{W}(Z)(Y_1, Y_2, Y_3)\|_{L^1(\Omega)} \leq Ch \|Y_1\|_{H^2_h(\Omega)}\|Y_2\|_{L^2_h(\Omega)}\|Y_3\|_{L^2_h(\Omega)}
	\label{Abels_(2.15)}
	\end{equation}
	for all $Y_1\in H^2(\Omega, \Rntimesn)$, $Y_2$, $Y_3\in L^2(\Omega; \Rntimesn)$, $0<h\leq 1$ and $\|Z\|_{L^\infty(\Omega} \leq \min\{\e, h\}$ and
	\begin{equation}
	\|D^3\tilde{W}(Z)(Y_1, Y_2, Y_3)\|_{L^1(\Omega)} \leq Ch \|Y_1\|_{H^1_h(\Omega)} \|Y_2\|_{H^1_h(\Omega)} \|Y_3\|_{L^2_h(\Omega)}
	\label{Abels_(2.16)}
	\end{equation}
	for all $Y_1$, $Y_2\in H^1(\Omega, \Rntimesn)$, $Y_3\in L^2(\Omega; \Rntimesn)$, $0<h\leq 1$ and $\|Z\|_{L^\infty(\Omega} \leq \min\{\e, h\}$ and 
	\begin{equation}
	\|D^3\tilde{W}(Z)(Y_1, Y_2, Y_3)\|_{L^1(\Omega)} \leq Ch \bigg\|\bigg(Y_1, \frac{1}{h} \operatorname{sym}(Y_1)\bigg)\bigg\|_{L^\infty(\Omega)} \|Y_2\|_{H^1_h(\Omega)} \|Y_3\|_{L^2_h(\Omega)}
	\label{globalBoundednessOfD3W(Z)(Y1,Y2,Y3)LInftyVersion}
	\end{equation}
	for all $Y_1\in L^\infty(\Omega, \Rntimesn)$, $Y_2$, $Y_3\in L^2(\Omega; \Rntimesn)$, $0<h\leq 1$ and $\|Z\|_{L^\infty(\Omega} \leq \min\{\e, h\}$.
	\label{Corollary_L1BoundsD3TildeW}
\end{Coro}
\begin{proof}
	The inequalities follow directly from Lemma \ref{Lemma_DecompD3TildeW} and Hölder's inequality.
\end{proof}

In order to bound the full scaled gradient $\nabla_h g$ of some function $g\in H^1_{per}(\Omega)$ by the symmetric one, we need a sharp Korn's inequality for thin rods. As rigid motions $x\mapsto \alpha x^\perp$ for $\alpha\in\R$ arbitrary are admissible functions in $H^1_{per}(\Omega)$ we can not expect that the full scaled gradient is bounded by $\e_h(g)$. Precisely we obtain the following results.
\begin{lem}
  There exists a constant $C=C(\Omega)>0$ such that for all $0 < h\leq 1$ and $u\in H^1_{per}(\Omega;\R^3)$ we have
  	\begin{equation}
	\bigg\|\nabla_h u - \frac{1}{h}B(u)\bigg\|_{L^2(\Omega)} \leq C\bigg\|\frac{1}{h} \e_h(u)\bigg\|_{L^2(\Omega)},
	\label{Korn}
	\end{equation}
        where
	\begin{equation}
	B(u) =
	\begin{pmatrix}
	0&0&0\\
	0&0& a(u)\\
	0& - a(u)& 0
	\end{pmatrix}
	\label{KornStructureS}
	\end{equation}
	with $a(u) = \frac{1}{|\Omega|}\int_\Omega \partial_{x_3} u_2(x) - \partial_{x_2} u_3(x) dx$.
	\label{KornIneqS}
\end{lem}
\begin{proof}
	The proof is similar to \cite[Lemma 2.1]{AbelsMoraMueller} and is done in \cite[Lemma 2.4.4]{AmeismeierDiss}
\end{proof}
\begin{lem}[Korn inequality in integral form]\ \\
	For all $0 < h\leq 1$ and $u\in H^1_{per}(\Omega;\R^3)$, there exists a constant $C_K=C_K(\Omega)$, such that
	\begin{equation}
	\|\nabla_h u\|_{L^2(\Omega)}\leq \frac{C_K}{h}\bigg( \|\e_h(u)\|_{L^2(\Omega)} + \bigg| \int_\Omega u \cdot x^\perp dx\bigg|\bigg)
	\label{KornMeanValue}
	\end{equation}
	where $x^\perp = (0, -x_3, x_2)^T$.
	\label{LemmaKornIntegralForm}
\end{lem}
\begin{proof}
	A proof can be found in \cite{AbelsAmeismeier1}.
\end{proof}

\section{First Order Expansion in a Linearised Regime}
We construct an approximation to the unique solution of the non-linear system
\begin{align}
\partial_t^2 u_h - \frac{1}{h^2}\operatorname{div}_h \Big(D\tilde{W}(\nabla_h u_h)\Big) &= h^{2} f_h \quad \text{in } \Omega\times [0, T), \label{NLS1}\\
D\tilde{W}(\nabla_h u_h)\nu|_{(0, L)\times \partial S} &= 0,\label{NLS2}\\
u_h \text{ is $L$-periodic} &\text{ w.r.t. } x_1, \label{NLS3}\\
(u_h, \partial_t u_h)|_{t=0} &= (u_{0,h}, u_{1,h}), \label{NLS4}
\end{align}
where $\tilde{W}(F) = W(Id+F)$ for all $F\in \Rtimes$, $T > 0$. We assume that 
\begin{equation*}
f^h(x,t) = \begin{pmatrix}
0 \\ g(x_1, t)
\end{pmatrix}
\end{equation*} 
for some $g\in \bigcap_{k=0}^3 W^k_1(0,T; H^{10-2k}_{per}(0,L; \R^2))$, which implies 
\begin{equation*}
\int_S f^h(x,t) x_k dx' = 0
\end{equation*} 
for $k=2, 3$. Moreover we assume that
\begin{equation}
\max_{\sigma = 0,1,2}\|\partial_t^\sigma g|_{t=0}\|_{H^{2-2\sigma}(0,L)} \leq M,
\label{lokalAssumptionOfSmallnessForRightHandSideG}
\end{equation} 
where $M>0$ is chosen later. Without loss of generality we can assume $\int_0^L g dx_1 = 0$. Otherwise we substract  
	\begin{equation*}
		a(t) := \frac1{|\Omega|}\left(\int_\Omega u_{0,h} dx - t\int_\Omega u_{1,h} \, dx - \int_0^t (t-s) \int_\Omega f^h(s)\, dx\, ds\right)
	\end{equation*}
from $u_h$ analogously as in the proof of \cite[Theorem~3.1]{AbelsAmeismeier1}.

\subsection{Construction of the ansatz function}
For the ansatz function we consider the following system of one-dimensional beam equations
\begin{gather*}
\partial_t^2 v + \begin{pmatrix}	I_2 & 0 \\ 0 & I_3 \end{pmatrix} \partial^4_{x_1} v = g,\\
v \text{ is $L$-periodic in } x_1,\\
(v,\partial_t v)|_{t=0} = (\tilde{v}_0, \tilde{v}_1),
\end{gather*}
where $\tilde{v}_{0}\in H^{12}_{per}(0,L;\R^2)$, $\tilde{v}_{1}\in H^{10}_{per}(0,L;\R^2)$ such that 
\begin{equation}
\|\tilde{v}_0\|_{H^8(0,L)} \leq M \quad\text{ and }\quad \|\tilde{v}_1\|_{H^5(0,L)} \leq M
\label{lokalAssumptionOfSmallnessForInitialData}
\end{equation}
and
\begin{equation*}
I_k := \int_S x_k^2 dx'\qquad \text{for }k=2,3. 
\end{equation*}
 Then we obtain with standard methods, as e.g. in \cite[Theorem 11.8]{RenardyRogers}, the existence of a unique solution
\begin{equation*}
v\in \bigcap_{j=0}^4 C^j([0,T]; H^{12-2j}_{per}(0,L;\R^2)).
\end{equation*}
Moreover, due to the assumptions for $g$ and the periodicity of $v$ it follows
\begin{equation*}
\partial_t^2 \int_0^L v dx_1 = 0.
\end{equation*}
Now we define
\begin{align}
\tilde{u}_h(x,t) &= h^2 \begin{pmatrix} 
0 \\ v_2 \\ v_3
\end{pmatrix} + h^3 \begin{pmatrix}
-x_2\partial_{x_1} v_{2} - x_3 \partial_{x_1} v_{3} \\ 0 \\ 0
\end{pmatrix} + h^5 \begin{pmatrix}
a_2(x') \partial_{x_1}^3 v_2 + a_3(x') \partial_{x_1}^3 v_3 \\ 0 \\ 0
\end{pmatrix}\notag\\ 
&+ h^6 \begin{pmatrix}
0 \\ b_2(x') \partial_{x_1}^4 v_2 + c_3(x') \partial_{x_1}^4 v_3 \\ b_3(x') \partial_{x_1}^4 v_3 + c_2(x') \partial_{x_1}^4 v_2
\end{pmatrix},\label{GlobalDefinitionOfTildeUH}
\end{align}
where $a$, $b$, $c\colon S\to\R^2$ are chosen later. Then
\begin{align*}
\nabla_h \tilde{u}_h(x,t) &= 
h^2 \begin{pmatrix}
0 & -\partial_{x_1} v_2 & -\partial_{x_1} v_3\\
\partial_{x_1} v_2 & 0 & 0\\
\partial_{x_1} v_3 & 0 & 0\\
\end{pmatrix} 
+ h^3 \begin{pmatrix}
-x_2 \partial_{x_1}^2 v_2 - x_3 \partial_{x_1}^2 v_3 & 0 & 0\\
0 & 0 & 0 \\
0 & 0 & 0
\end{pmatrix}\\
& \quad + h^4 \begin{pmatrix}
0 & \partial_{x_2} a_2 \partial_{x_1}^3 v_2 + \partial_{x_2} a_3 \partial_{x_1}^3 v_3 & \partial_{x_3} a_2 \partial_{x_1}^3 v_2 + \partial_{x_3} a_3\partial_{x_1}^3 v_3 \\
0&0&0\\
0&0&0\end{pmatrix} \\
& \quad +h^5 \begin{pmatrix}
a_2 \partial_{x_1}^4 v_2 + a_3 \partial_{x_1}^4 v_3 & 0 &0\\
0& \partial_{x_2} b_2 \partial_{x_1}^4 v_2 + \partial_{x_2} c_3 \partial_{x_1}^4 v_3  & \partial_{x_3} b_2 \partial_{x_1}^4 v_2 +\partial_{x_3} c_3 \partial_{x_1}^4 v_3\\
0 & \partial_{x_2} b_3 \partial_{x_1}^4 v_3 +\partial_{x_2} c_2 \partial_{x_1}^4 v_2 & \partial_{x_3} b_3 \partial_{x_1}^4 v_3 + \partial_{x_3} c_2 \partial_{x_1}^4 v_2
\end{pmatrix}\\
& \quad + h^6 \begin{pmatrix}
0 & 0 & 0 \\
b_2 \partial_{x_1}^5 v_2 + c_3 \partial_{x_1}^5 v_3 & 0 & 0 \\
b_3 \partial_{x_1}^5 v_3 + c_2 \partial_{x_1}^5 v_2  & 0 & 0
\end{pmatrix}.
\end{align*}
Thus with $D^2 W(Id) F = \operatorname{sym} F$ we can derive
\begin{align*}
\frac{1}{h^2} \divh (D^2 W(Id) \nabla_h \tilde{u}_h) & =  h \begin{pmatrix}
\big(\frac{1}{2} \Delta a - (x_2, x_3)^T\big)\cdot \partial_{x_1}^3 v \\ 0 \\0
\end{pmatrix}\\
& \quad + h^2 \begin{pmatrix}
0 \\ \nabla_{x'} a(x')^T \partial_{x_1}^4 v +\begin{pmatrix}
\partial_{x_2}^2 b_2 & \partial_{x_2}^2 c_3\\
\partial_{x_3}^2  c_2 & \partial_{x_3}^2 b_3
\end{pmatrix} \partial_{x_1}^4 v
\end{pmatrix}\\
& \quad + \frac{h^2}{2} \begin{pmatrix}
0 \\ \begin{pmatrix}
\partial_{x_3} \partial_{x_2} c_2 + \partial_{x_3}^2 b_2 & \partial_{x_2}\partial_{x_3} b_3 + \partial_{x_3}^2 c_3 \\
\partial_{x_2}^2 c_2  + \partial_{x_2}\partial_{x_3} b_2 & \partial_{x_2}^2 b_3 + \partial_{x_2}\partial_{x_3} c_3
\end{pmatrix} \partial_{x_1}^4 v
\end{pmatrix} + r_h(x,t)
\end{align*}
for
\begin{equation*}
r_h(x,t) = O(h^3).
\end{equation*}
Moreover for the boundary condition it holds
\begin{align*}
D^2 & W(Id)[\nabla_h \tilde{u}_h] \nu = h^4 \begin{pmatrix}
\frac{1}{2}(\nabla_{x'} a \nu_{\partial S}) \cdot \partial_{x_1}^3 v \\ 0 \\ 0
\end{pmatrix}\\
& \quad + h^5 \begin{pmatrix}
0 \\ \big(\partial_{x_2} b_2 \nu_2 + \frac{1}{2}(\partial_{x_2} c_2 + \partial_{x_3} b_2)\nu_3\big) \partial_{x_1}^4 v_2 + \big(\partial_{x_2} c_3 \nu_2 + \frac{1}{2}(\partial_{x_2} b_3 + \partial_{x_3} c_3) \nu_3) \partial_{x_1}^4 v_3\big)\\ \big(\frac{1}{2}(\partial_{x_2} c_2 + \partial_{x_3} b_2) \nu_2 + \partial_{x_3} c_2 \nu_3\big) \partial_{x_1}^4 v_2 + \Big(\frac{1}{2}(\partial_{x_2} b_3 + \partial_{x_3} c_3)\nu_2 + \partial_{x_3} b_3 \nu_3\Big) \partial_{x_1}^4 v_3
\end{pmatrix}\\
& \quad + \frac{h^6}{2} \begin{pmatrix}
\nu^T
\begin{pmatrix}
b_2 & c_2 \\ b_3 & c_3
\end{pmatrix} \partial_{x_1}^5 v \\ 0 \\ 0
\end{pmatrix}\\
& =  h^4 \begin{pmatrix}
(\nabla_{x'} a \nu_{\partial S}) \cdot \partial_{x_1}^3 v \\ 0 \\ 0
\end{pmatrix} + h^5 \begin{pmatrix}
0 \\  \nu^T \begin{pmatrix}
\partial_{x_2} b_2 & \frac{1}{2}(\partial_{x_2} c_2 + \partial_{x_3} b_2) \\ \partial_{x_2} c_3  & \frac{1}{2} (\partial_{x_2} b_3 + \partial_{x_3} c_3)
\end{pmatrix} \partial_{x_1}^4 v \\ 
\nu^T 
\begin{pmatrix}
\frac{1}{2}(\partial_{x_2} c_2 + \partial_{x_3} b_2) & \partial_{x_3} c_2 \\ \frac{1}{2}(\partial_{x_2} b_3 + \partial_{x_3} c_3) & \partial_{x_3} b_3
\end{pmatrix}\partial_{x_1}^4 v
\end{pmatrix}\\
& \quad + \frac{h^6}{2} \begin{pmatrix}
\nu^T \begin{pmatrix}
b_2 & c_2 \\ b_3 & c_3
\end{pmatrix} \partial_{x_1}^5 v \\ 0 \\ 0
\end{pmatrix}.
\end{align*}
We choose now $a\colon S\to\R^2$ as the solution of the following system
\begin{align*}
\left\{
\begin{aligned}
-\Delta a &= -2\begin{pmatrix} x_2 \\ x_3 \end{pmatrix} \quad&&\text{ in } S\\
\nabla_{x'} a \nu &= 0 \quad&&\text{ on } \partial S
\end{aligned}
\right.
\end{align*}
with 
\begin{equation*}
\int_S a(x') dx' = 0.
\end{equation*}
Such a solution exists, because we can apply the Lax-Milgram Lemma for the weak Laplacian on $H^1_{(0)}(S;\R^2)$. Thereby, the coercivity follows from Poincaré's inequality. With well known regularity result, e.g. Theorem 4.18 in \cite{McLean}, we obtain $a\in C^\infty(\overline{S},\R^2)$. The systems for $b$ and $c$ decouple to 
\begin{align}
\left\{
\begin{aligned}
\partial_{x_2}^2 b_2 + \frac{1}{2} \partial_{x_3}^2 b_2 + \frac{1}{2}\partial_{x_3}\partial_{x_2} c_2 &= I_1 - \partial_{x_2} a_2 &&\quad\text{ in } S\\
\frac{1}{2} \partial_{x_2}^2 c_2 + \partial_{x_3}^2 c_2 + \frac{1}{2} \partial_{x_2} \partial_{x_3} b_2 & = -\partial_{x_3} a_2 &&\quad\text{ in } S
\end{aligned}
\right.
\label{SystemForC2andB2}
\end{align}
and
\begin{align}
\left\{
\begin{aligned}
\partial_{x_2}^2 c_3 + \frac{1}{2} \partial_{x_3}^2 c_3 + \frac{1}{2}\partial_{x_2}\partial_{x_3} b_3 &= - \partial_{x_2} a_3 &&\quad\text{ in } S\\
\frac{1}{2} \partial_{x_2}^2 b_3 + \partial_{x_3}^2 b_3 + \frac{1}{2} \partial_{x_2} \partial_{x_3} c_3 & = I_2 -\partial_{x_3} a_3 &&\quad\text{ in } S
\end{aligned}
\right.
\label{SystemForC3andB3}
\end{align}
Defining the matrix of coefficients $(\mathfrak{p}^{\alpha \beta}_{ij})^{\alpha, \beta = 2,3}_{i,j=1,2}$ in the following way
\begin{gather*}
\mathfrak{p}^{22}_{11} = 1\quad \mathfrak{p}^{33}_{11} = \frac{1}{2}\quad \mathfrak{p}^{32}_{12} = \frac{1}{4} \quad \mathfrak{p}^{23}_{12} = \frac{1}{4}\\
\mathfrak{p}^{22}_{22} = \frac{1}{2}\quad \mathfrak{p}^{33}_{22} = 1\quad \mathfrak{p}^{23}_{21} = \frac{1}{4} \quad \mathfrak{p}^{32}_{21} = \frac{1}{4}\\
\mathfrak{p}^{\alpha \beta}_{ij} = 0 \text{ otherwise}.
\end{gather*}
With $w = (b_2, c_2)^T$ and $f = (-I_1 - \partial_{x_2} a_2, -\partial_{x_3} a_2)^T$, \eqref{SystemForC2andB2} is equivalent to
\begin{equation*}
\sum_{\alpha,\beta=2}^3 \sum_{j=1}^2 -\partial_\beta \big(\mathfrak{p}^{\alpha\beta}_{ij} \partial_\alpha w_j\big) = f_i
\end{equation*}	
for $i=1$, $2$. Let now
\begin{equation*}
\xi := \begin{pmatrix} \xi_{12} & \xi_{13} \\ \xi_{22} & \xi_{23} \end{pmatrix}\in \R^{2\times 2}.
\end{equation*}
be arbitrary. Then it holds
\begin{align*}
\sum_{\alpha,\beta=2}^3 \sum_{i,j=1}^2 \mathfrak{p}^{\alpha\beta}_{ij} \xi_{i\alpha} \xi_{j\beta} 
&= \frac{3}{4}(\xi_{12}^2 + \xi_{23}^2) + \frac{1}{4} (\xi_{12} + \xi_{23})^2 + \frac{1}{4}(\xi_{13}^2 + \xi_{22}^2) + \frac{1}{4} (\xi_{13} + \xi_{22})^2 \\
&\geq \frac{1}{4} (\xi_{12}^2 + \xi_{13}^2 + \xi_{22}^2 + \xi_{23}^2) = \frac{1}{4} |\xi|^2
\end{align*}
and thus $\mathfrak{p}^{\alpha\beta}_{ij}$ satisfies the Legendre condition for $\lambda = \frac{1}{4}$.
Thus we can solve \eqref{SystemForC2andB2} and \eqref{SystemForC3andB3} with homogeneous Dirichlet boundary condition
\begin{equation*}
\begin{pmatrix}	b_2 \\  c_2 \end{pmatrix} = 0 \quad\text{and}\quad \begin{pmatrix} b_3 \\ c_3 \end{pmatrix} = 0 \quad\text{ on } \partial S
\end{equation*}
as the system \eqref{SystemForC3andB3} can be treated in the same manner. The regularity of $a$ implies now that $b = (b_2,b_3)$ and $c=(c_2,c_3)$ are $C^\infty(\overline{S};\R^2)$.

The approximating solution $\tilde{u}_h$ solves then the following system
\begin{alignat*}{2}
\partial_t^2 \tilde{u}_h - \frac{1}{h^2} \divh\Big(D^2\tilde{W}(0) \nabla_h\tilde{u}_h\Big) &= h^{2} f_h - r_h \quad &&\text{ in } \Omega\times (0,T),\\
D^2\tilde{W}(0)[\nabla_h \tilde{u}_h]\nu\Big|_{(0,L)\times \partial S} &= \tr{\partial\Omega}(r_{N,h}) \nu \quad &&\text{ on } \partial\Omega\times (0,T),\\
\tilde{u}_h \text{ is } L\text{-periodic}&\text{ in } x_1 \text{-direction},&&\\
(\tilde{u}_h, \partial_t \tilde{u}_h)|_{t=0} &= (\tilde{u}_{0,h}, \tilde{u}_{1,h}), &&
\end{alignat*}
where $r_h$ is chosen as above,
\begin{equation*}
r_{N,h} := h^5 \begin{pmatrix}
0 \\  \nu^T \begin{pmatrix}
\partial_{x_2} b_2 & \frac{1}{2}(\partial_{x_2} c_2 + \partial_{x_3} b_2) \\ \partial_{x_2} c_3  & \frac{1}{2} (\partial_{x_2} b_3 + \partial_{x_3} c_3)
\end{pmatrix} \partial_{x_1}^4 v \\ 
\nu^T 
\begin{pmatrix}
\frac{1}{2}(\partial_{x_2} c_2 + \partial_{x_3} b_2) & \partial_{x_3} c_2 \\ \frac{1}{2}(\partial_{x_2} b_3 + \partial_{x_3} c_3) & \partial_{x_3} b_3
\end{pmatrix}\partial_{x_1}^4 v
\end{pmatrix},
\end{equation*}
and the initial data is given by
\begin{align}
\tilde{u}_{j,h}(x,t) &= h^2 \begin{pmatrix} 
0 \\ v^j_2 \\ v^j_3
\end{pmatrix} + h^3 \begin{pmatrix}
-x_2 \partial_{x_1} v^j_{2} - x_3 \partial_{x_1} v^j_{3} \\ 0 \\ 0
\end{pmatrix} + h^5 \begin{pmatrix}
a_2(x') \partial_{x_1}^3 v^j_2 + a_3(x') \partial_{x_1}^3 v^j_3 \\ 0 \\ 0
\end{pmatrix}\notag\\ 
&\qquad + h^6 \begin{pmatrix}
0 \\ b_2(x') \partial_{x_1}^4 v^j_2 + c_3(x') \partial_{x_1}^4 v^j_3 \\ b_3(x') \partial_{x_1}^4 v^j_3 + c_2(x') \partial_{x_1}^4 v^j_2
\end{pmatrix} \label{GlobalDefinitionOfInitialValuesOfTildeUH}
\end{align}
with $v^j := \partial_t^j v|_{t=0}$ and $j=0,\ldots, 4$. For the remainder it holds
\begin{equation*}
\|r_h\|_{C^0([0,T];L^2)} \leq Ch^3\quad\text{ and }\quad\|r_{N,h}\|_{C^2([0,T]; H^1)} \leq Ch^5.
\end{equation*}

\subsection{Existence of and Bounds on Initial Values}
\label{subsection::ExistenceOfAndBoundsOnInitialValues}
Define now
\begin{equation*}
\mathcal{B} := H^1_{per}(\Omega;\R^3)\cap \Big\{u\in L^2(\Omega;\R^3) \;:\; \int_\Omega udx = \int_\Omega u\cdot x^\perp dx = 0 \Big\}
\end{equation*}
equipped with the norm
\begin{equation*}
\|u\|_{\mathcal{B}_h} := \bigg\|\frac{1}{h} \e_h(u)\bigg\|_{L^2(\Omega)}.
\end{equation*}
\begin{lem}
	 There exists constants $C_0>0$ and $M_0\in (0,1]$ such that for $0<h\le 1$ and $f\in H^{1,1}_{per}(\Omega;\R^3)$ with $\|f\|_{H^{1,1}(\Omega)} \leq M_0 h$ and $\int_\Omega f dx = 0$ there exists a unique solution $w\in H^3_{per}(\Omega; \R^3)\cap \mathcal{B}$ with $\partial_{x_1} w\in H^3_{per}(\Omega;\R^3)$ of 
	\begin{equation}
	\frac{1}{h^2}\Big(D\tilde{W}(\nabla_h w), \nabla_h \varphi \Big)_{L^2(\Omega)} = (f,\varphi)_{L^2(\Omega)}\quad \text{for all }\varphi \in \mathcal{B}.
	\label{EquationDefiningU0H}
	\end{equation}
	 Moreover
	\begin{equation}
	\bigg\|\bigg(\frac{1}{h}\e_h(w), \nabla_h\frac{1}{h}\e(w), \nabla_h^2 w\bigg)\bigg\|_{H^{1,1}(\Omega)} \leq C_0 \|f\|_{H^{1,1}(\Omega)}
	\label{InequalityBoundsForU0H}
	\end{equation}
	holds. If $w'\in H^3_{per}(\Omega;\R^3)\cap \mathcal{B}$ with $\partial_{x_1} w\in H^3_{per}(\Omega;\R^3)$ is the solution to $f'\in H^{1,1}_{per}(\Omega;\R^3)$ with $\|f'\|_{H^{1,1}(\Omega)} \leq M_0 h$ and $\int_\Omega f' dx = 0$, then it holds
	\begin{equation}
	\bigg\|\bigg(\frac{1}{h}\e_h(w - w'), \nabla_h\frac{1}{h}\e(w - w'), \nabla_h^2 (w - w')\bigg)\bigg\|_{H^{1,1}(\Omega)} \leq C_0 \|f - f'\|_{H^{1,1}(\Omega)}.
	\label{InequalityBoundsForDifferenzOfSolutionsForNonlinearEquation}
	\end{equation}
	\label{LemmaExistenceOfU0H}
\end{lem}
\begin{proof}
	Using a Taylor series expansion for $D\tilde{W}(\nabla_h w)$ we obtain
	\begin{align}
	D\tilde{W}(\nabla_h w) &= D\tilde{W}(0) + D^2\tilde{W}(0)[\nabla_h w] + \int_0^1 (1-\tau) D^3\tilde{W}(\tau\nabla_h w)[\nabla_h w, \nabla_h w]d\tau \notag\\
	& =: D^2\tilde{W}(0)\nabla_h w + G(\nabla_h w). \label{GlobalDefinitionOfNonlinearityG}
	\end{align}
	Thus \eqref{EquationDefiningU0H} is equivalent to
	\begin{equation*}
	\langle L_h w, \varphi\rangle_{\mathcal{B}', \mathcal{B}} :=\frac{1}{h^2}\Big(D^2\tilde{W}(0)\nabla_h w, \nabla_h \varphi\Big)_{L^2(\Omega)}	= (f,\varphi)_{L^2(\Omega)} - \frac{1}{h^2}(G(\nabla_h w), \nabla_h \varphi)_{L^2(\Omega)}.
	\end{equation*} 
	The idea is now to use the contraction mapping principle in order to prove the existence of a solution for \eqref{EquationDefiningU0H}, i.e., with the later equivalence
	\begin{equation*}
	w = \mathcal{G}_{h,f} (w) := L_h^{-1}\Big(f, G_h(w)\Big)
	\end{equation*}
	holds with $G_h(w) := \frac{1}{h^2} G(\nabla_h w)$. Consequently we investigate the mapping properties of $L_h$ and $G_h$.
	
	For $f\in L^2(\Omega;\R^3)$ and $F\in L^2(\Omega;\Rtimes)$ we obtain with the Lemma of Lax-Milgram the existence of a unique solution $w\in \mathcal{B}$ for 
	\begin{equation}
	\langle L_h w, \varphi\rangle_{\mathcal{B}',\mathcal{B}} = (f,\varphi)_{L^2(\Omega)} - (F, \nabla_h \varphi)_{L^2(\Omega)}
	\label{LocalEquationLinearSolutionOperator}
	\end{equation}
	for all $\varphi\in\mathcal{B}$.
	The solution satisfies 
	\begin{equation*}
	\|w\|_{\mathcal{B}} = \bigg\|\frac{1}{h} \e_h(w)\bigg\|_{L^2(\Omega)} \leq C \big( \|f\|_{L^2(\Omega)} + \|F\|_{(L^2_h)'}\big).
	\end{equation*}
	If now $f\in H^{0,k}(\Omega;\R^3)$ and $F\in H^{0,k}(\Omega;\Rtimes)$ for $k=1,2$, it follows by a different quotient argument that $w\in H^{0,k}(\Omega;\R^3)$ holds and
	\begin{equation}
	\bigg\|\frac{1}{h}\e_h(w)\bigg\|_{H^{0,k}(\Omega)} \leq C \Big(\|f\|_{H^{0,k-1}(\Omega)} +  \max_{j=0,\ldots, k}\|\partial_{x_1}^j F\|_{(L^2_h)'}\Big).
	\label{LocalInequalityAnisotropicRegularityBoundRightHandSidefandF}
	\end{equation}
	Using the decomposition $\B \oplus \operatorname{span}\{x\mapsto x^\perp\} = H^1_{(0),per}(\Omega;\R^3)$ it follows that for
	\begin{equation*}
	\alpha := (F, \nabla_h x^\perp)_{L^2(\Omega)} - (f,x^\perp)_{L^2(\Omega)} 
	\end{equation*}
	we have 
	\begin{equation*}
	\frac{1}{h^2}\Big(D^2\tilde{W}(0)\nabla_h w, \nabla_h \varphi\Big)_{L^2(\Omega)} =  (f + \alpha x^\perp,\varphi)_{L^2(\Omega)} - (F, \nabla_h \varphi)_{L^2(\Omega)}
	\end{equation*}
	for all $\varphi \in H^1_{(0),per}(\Omega;\R^3)$. Hence, if $f\in H^1_{per}(\Omega;\R^3)$ and $F\in H^2_{per}(\Omega;\Rtimes)$, then $w$ solves the system 
	\begin{align*}
	\left\{
	\begin{aligned}
	- \frac{1}{h^2} \operatorname{div}_h (D^2\tilde{W}(0)\nabla_h w) &= f + \alpha x^\perp - \operatorname{div}_h F &&\quad\text{in } \Omega\\
	D^2\tilde{W}(0)[\nabla_h w]\nu\Big|_{\partial S} &= h^2 \tr{\partial \Omega}(F)\nu\Big|_{\partial S} &&\quad\text{in }  \partial \Omega
	\end{aligned}
	\right.
	\end{align*}
	in a weak sense. Thus with elliptic regularity theory it follows $w\in H^3_{per}(\Omega;\R^3)\cap \mathcal{B}$. By Theorem \ref{TheoremHigherRegularity} in the appendix, we obtain
	\begin{align*}
	\bigg\|\bigg(\frac{1}{h}\e_h(w), \nabla &\frac{1}{h}\e_h(w), \nabla_h^2 w\bigg)\bigg\|_{H^1(\Omega)}  \leq C\bigg(h^2\|(f,\operatorname{div}_h F)\|_{H^1(\Omega)} + \|f\|_{H^{0,1}(\Omega)}\\
	& + \max_{j=0,1,2}\|\partial_{x_1}^j F\|_{(L^2_h)'} + \Big\|h \tr{\partial \Omega}(F) \Big\|_{L^2(0,L; H^\frac{3}{2}(\partial S))\cap H^1(0,L;H^\frac{1}{2}(\partial S))} \bigg),\notag
	\end{align*}
	where we have exploited
	\begin{equation*}
	h^2 |\alpha| \leq Ch^2 \|f\|_{L^2(\Omega)} + Ch \|F\|_{(L^2_h)'}.
	\end{equation*}
	Using that $\tr{\partial S}\colon H^2(S) \to H^{\frac{3}{2}}(\partial S)$ is a bounded operator we obtain
	\begin{align*}
	h \Big\|\tr{\partial \Omega}(F) \Big\|_{L^2(0,L; H^\frac{3}{2}(\partial S))\cap H^1(0,L;H^\frac{1}{2}(\partial S))} &\leq Ch \Big(\|F\|_{H^{1,1}(\Omega)} + \max_{k=0,1,2}\|\nabla_{x'}^k F\|_{L^2(\Omega)} \Big)\\
	& \leq C\Big(\max_{j=0,1,2}\|\partial_{x_1}^j F\|_{(L^2_h)'} + h^2 \|\nabla_h F\|_{H^1(\Omega)}\Big)
	\end{align*}
	because of
	\begin{equation*}
	\|F\|_{H^1(\Omega)} \leq \|F\|_{H^{0,1}(\Omega)} + \|\nabla_{x'} F\|_{L^2(\Omega)}\;\text{ and }\; \|F\|_{L^2(\Omega)} \leq \frac{1}{h} \|F\|_{(L^2_h)'}.
	\end{equation*}
	Thus we deduce for some $C_L > 0$
	\begin{align}
	\bigg\|\bigg(\frac{1}{h}\e_h(w), \nabla &\frac{1}{h}\e_h(w), \nabla_h^2 w\bigg)\bigg\|_{H^1(\Omega)}\notag \\
	& \leq C_L \Big(h^2\|(f, \nabla_h F)\|_{H^1(\Omega)} + \|f\|_{H^{0,1}(\Omega)} + \max_{j=0,1,2}\|\partial_{x_1}^j F\|_{(L^2_h)'}\Big).
	\label{LocalInequalityHigherRegularityRightHandSidefandF}
	\end{align}
	We define $\mathcal{X}_h := H^3_{per}(\Omega;\R^3)\cap \mathcal{B}$ and $\mathcal{Y}_h := H^1_{per}(\Omega;\R^3) \times H^2_{per}(\Omega;\Rtimes)$ normed via 
	\begin{gather*}
	\|g\|_{\mathcal{X}_h} := \bigg\|\bigg(\frac{1}{h}\e_h(g), \nabla\frac{1}{h}\e_h(g), \nabla_h^2 g\bigg)\bigg\|_{H^1(\Omega)}\\
	\|(f,F)\|_{\mathcal{Y}_h} := h^2\|(f, \nabla_h F)\|_{H^1(\Omega)} + \|f\|_{H^{0,1}(\Omega)} + \max_{j=0,1,2}\|\partial_{x_1}^j F\|_{(L^2_h)'}.
	\end{gather*}
	With this $L_h^{-1}\colon \mathcal{Y}_h \to \mathcal{X}_h$ is a bilinear, bijective and bounded operator, mapping a tuple $(f,F)\in \mathcal{Y}_h$ to the corresponding solution $w\in\mathcal{X}_h$ of \eqref{LocalEquationLinearSolutionOperator}. In order to close the proof we have to show that $G_h$ is a  contraction  with respect to the relevant norms.
	
	In a first step we assume that $w_i\in\mathcal{X}_h$ with 
	\begin{equation*}
	\|w_i\|_{\mathcal{X}_h} \leq C_0 M_1 h
	\end{equation*}
	for $i=1,2$ and $M_1 > 0$ to be chosen later. Then
	\begin{align*}
	\|G_h (w_1) & - G_h(w_2)\|_{(L^2_h)'} = \bigg\|\frac{1}{h^2} \int_0^1 (1-\tau) \Big(D^3\tilde{W}(\tau\nabla_h w_1)[\nabla_h w_1 - \nabla_h w_2, \nabla_h w_1]\\
	&\quad + D^3\tilde{W}(\tau\nabla_h w_2)[\nabla_h w_1 - \nabla_h w_2, \nabla_h w_2]\\
	&\quad + \big(D^3\tilde{W}(\tau\nabla_h w_1) - D^3\tilde{W}(\tau\nabla_h w_2)\big)[\nabla_h w_1, \nabla_h w_2]\Big) d\tau \bigg\|_{(L^2_h)'}\\
	& \leq CM_1 \|\nabla_h (w_1 - w_2)\|_{H^1_h(\Omega)}\\
	& \quad + \bigg\|\frac{1}{h^2} \int_0^1(1-\tau) \int_0^1 Q(\tau, t, w_1, w_2) dt [\tau(\nabla_h w_1 - \nabla_h w_2), \nabla_h w_1, \nabla_h w_2]d\tau\bigg\|_{(L^2_h)'}\\
	& \leq C M_1 \bigg\|\frac{1}{h}\e_h(w_1 - w_2)\bigg\|_{\mathcal{X}_h},
	\end{align*}
	where we used Corollary \ref{Corollary_L1BoundsD3TildeW}, $\|\nabla_h w_j\|_{H^1_h(\Omega)}\leq C\|w_j\|_{\mathcal{X}_h}$ and the boundedness of
	\begin{equation*}
	Q(\tau, t, w_1, w_2) := D^4\tilde{W}(t\tau \nabla_h w_1 + (1-t)\tau \nabla_h w_2).
	\end{equation*}
	The definition of $G$ implies that for $k=1,2,3$ it holds
	\begin{align}
	\partial_{x_k} G(\nabla_h w) &= D^2\tilde{W}(\nabla_h w)[\nabla_h \partial_{x_k} w] - D^2\tilde{W}(0) [\nabla_h \partial_{x_k} w]\label{LocalEqualityFirstDerivativeOfG} \\
	& = \int_0^1 D^3\tilde{W}(\tau\nabla_h w) [\nabla_h w, \nabla_h\partial_{x_k} w] d\tau.\notag
	\end{align}
	Hence, analogously as above
	\begin{align*}
	\|\partial_{x_k} &(G_h(w_1) - G_h(w_2)\|_{(L^2_h)'} \leq \bigg\| \frac{1}{h^2} \int_0^1 D^3\tilde{W}(\tau\nabla_h w_1) [\nabla_h (w_1 - w_2), \nabla_h \partial_{x_k} w_1]d\tau\bigg\|_{(L^2_h)'}\\
	& \quad + \bigg\| \frac{1}{h^2} \int_0^1 D^3\tilde{W}(\tau\nabla_h w_2) [\nabla_h w_2, \nabla_h \partial_{x_k} (w_1 - w_2)]d\tau\bigg\|_{(L^2_h)'}\\
	& \quad + \bigg\| \frac{1}{h^2} \int_0^1 \big(D^3\tilde{W}(\tau\nabla_h w_1) - D^3\tilde{W}(\tau\nabla_h w_2)\big) [\nabla_h w_2, \nabla_h \partial_{x_k} w_1]d\tau \bigg\|_{(L^2_h)'}\\
	& \leq \frac{C}{h}\|\nabla_h (w_1 - w_2)\|_{H^2_h(\Omega)} \|\nabla_h \partial_{x_k} w_1\|_{L^2_h(\Omega)} + \frac{C}{h} \|\nabla_h w_2\|_{H^2_h(\Omega)} \|\nabla_h \partial_{x_k}(w_1 - w_2)\|_{(L^2_h)'}\\
	&\quad + CM_1 \bigg\|\frac{1}{h}\e_h(w_1 - w_2)\bigg\|_{H^1_h(\Omega)}\\
	& \leq CM_1 \|w_1 - w_2\|_{\mathcal{X}_h}
	\end{align*}
	as
	\begin{equation*}
	\|\nabla_h \partial_{x_k} \varphi\|_{L^2_h(\Omega)} \leq \|\nabla_h \partial_{x_k} \varphi\|_{L^2(\Omega)} + \bigg\|\frac{1}{h} \e_h(\partial_{x_k} \varphi)\bigg\|_{L^2(\Omega)} \leq \bigg\|\bigg(\nabla \frac{1}{h}\e_h(\varphi), \nabla_h^2 \varphi \bigg)\bigg\|_{L^2(\Omega)} \leq \|\varphi\|_{\mathcal{X}_h}
	\end{equation*}
	for $\varphi = w_1$ and $\varphi = w_1 - w_2$. With the aid of \eqref{LocalEqualityFirstDerivativeOfG} it follows for $j, k=1, 2, 3$
	\begin{align*}
	\partial_{x_j}\partial_{x_k} G(\nabla_h w) &= D^2\tilde{W}(\nabla_h u_h)[\nabla_h\partial_{x_j}\partial_{x_k} w] - D^2\tilde{W}(0)[\nabla_h\partial_{x_j}\partial_{x_k} w] \\
	& \quad + D^3\tilde{W}(\nabla_h w) [\nabla_h \partial_{x_j} w, \nabla_h \partial_{x_k} w]\\
	& = \int_0^1 D^3\tilde{W}(\tau\nabla_h w)[\nabla_h w, \nabla_h\partial_{x_j}\partial_{x_k} w] d\tau + D^3\tilde{W}(\nabla_h w) [\nabla_h \partial_{x_j} w, \nabla_h \partial_{x_k} w].
	\end{align*}
	Thus we obtain in the same manner as above 
	\begin{equation*}
	\|\partial_{x_j}\partial_{x_k} (G_h(w_1) - G_h(w_2))\|_{(L^2_h)'}
	 \leq CM_1 \|w_1 - w_2\|_{\mathcal{X}_h}.
	\end{equation*}
	The fact that $h^2\|\nabla_h F\|_{H^1(\Omega)} \leq h \|\nabla F\|_{H^1(\Omega)}$ and $\|F\|_{L^2(\Omega)} \leq \frac{1}{h}\|F\|_{(L^2_h)'}$ implies with the later estimates that for $M_1\in (0,1]$ small enough
	\begin{align*}
	\mathcal{G}_{h,f}\colon\overline{B_{CM_1 h}(0)} \subset\mathcal{X}_h \to\mathcal{X}_h
	\end{align*}
	is a $\frac{1}{2}$-contraction. The self-mapping property of $\mathcal{G}_{h,f}$ follows because of
	\begin{equation*}
	\|\mathcal{G}_{h,f} (0)\|_{\mathcal{X}_h} = \|L^{-1}(f,0)\|_{\mathcal{X}_h} \leq C_L \|(f,0)\|_{\mathcal{Y}_h} \leq C_L \|f\|_{H^{1,1}(\Omega)} \leq C_L M_0 h.
	\end{equation*}
	Thus we can choose $M_0 > 0$ so small that $C_L M_0 h \leq \frac{C M_1}{2}$. Then we obtain with the $\frac{1}{2}$-contraction property of $\mathcal{G}_{h,f}$ for $w\in \overline{B_{CM_1 h}(0)}$
	\begin{align*}
	\|\mathcal{G}_{h,f}(w)\|_{\mathcal{X}_h} \leq \|\mathcal{G}_{h,f}(w) - \mathcal{G}_{h,f}(0)\|_{\mathcal{X}_h} + \|\mathcal{G}_{h,f}(0)\|_{\mathcal{X}_h} \leq \frac{1}{2} \|w\|_{\mathcal{X}_h} + C_L M_0 h \leq C M_1 h. 
	\end{align*}
	Therefore \eqref{InequalityBoundsForU0H} and \eqref{InequalityBoundsForDifferenzOfSolutionsForNonlinearEquation} hold with the $H^{1,1}(\Omega)$-norm on the left hand side replaced by the $\mathcal{X}_h$-norm.
	
	Using the decomposition $\B \oplus \operatorname{span}\{x\mapsto x^\perp\} = H^1_{(0),per}(\Omega;\R^3)$ it follows that for
	\begin{equation*}
	\rho := \frac{1}{\mu(S) h^2} \Big(D\tilde{W}(\nabla_h w), \nabla_h x^\perp\Big)_{L^2(\Omega)}
	\end{equation*}
	we have 
	\begin{equation*}
	\frac{1}{h^2}\Big(D\tilde{W}(\nabla_h w), \nabla_h \varphi \Big)_{L^2(\Omega)} = (f - \rho x^\perp,\varphi)_{L^2(\Omega)}
	\end{equation*}
	for all $\varphi \in H^1_{(0),per}(\Omega;\R^3)$. If now $f\in H^{1,1}_{per}(\Omega;\R^3)$ we obtain, with a difference quotient argument, that $w\in H^3_{per}(\Omega;\R^3)\cap\mathcal{B}$ satisfies
	\begin{equation*}
	\frac{1}{h^2}\Big(D^2\tilde{W}(\nabla_h w) \nabla_h\partial_{x_1} w, \nabla_h \varphi\Big)_{L^2(\Omega)} = (\partial_{x_1} f, \varphi)_{L^2(\Omega)}.
	\end{equation*}
	for all $\varphi\in H^1_{(0),per}(\Omega;\R^3)$. Thus with Theorem \ref{TheoremHigherRegularity} the claimed inequalities follow.
\end{proof}
We define the initial values for the analytical problem as 
\begin{equation*}
u_{2+j, h} :=h^{2} \begin{pmatrix} 
0 \\ v^{2+j}_2 \\ v^{2+j}_3
\end{pmatrix} + h^{3} \begin{pmatrix}
-x_2 \partial_{x_1}v^{2+j}_{2} - x_3 \partial_{x_1}v^{2+j}_{3} \\ 0 \\ 0
\end{pmatrix}
\end{equation*}
for $j=1,2$ and $v^{2+j} = \partial_t^{2+j} v|_{t=0}$ as above.
\begin{lem}
	Let $\tilde{u}_h$ be as in \eqref{GlobalDefinitionOfTildeUH}, $\tilde{u}_{j,h}$ for $j=0,1,2$ as in \eqref{GlobalDefinitionOfInitialValuesOfTildeUH}, $u_{3,h}$, $u_{4,h}$ and $f^h$ be as above. Then for sufficiently small $h_0\in (0,1]$ and $M>0$ there exist solutions $(u_{0,h}, u_{1,h}, u_{2,h})$ of
	\begin{align}
	\frac{1}{h^2} \Big(D\tilde{W}(\nabla_h u_{0,h}), \nabla_h\varphi\Big)_{L^2(\Omega)} &= (h^2 f_h|_{t=0}, \varphi)_{L^2(\Omega)} - (u_{2,h}, \varphi)_{L^2(\Omega)}\label{GlobalDefiningEquationForU0H}\\
	\frac{1}{h^2} \Big(D^2\tilde{W}(\nabla_h u_{0,h})\nabla_h u_{1,h}, \nabla_h\varphi\Big)_{L^2(\Omega)} &= (h^2 \partial_t f|_{t=0}, \varphi)_{L^2(\Omega)} - (u_{3,h}, \varphi)_{L^2(\Omega)} \label{GlobalDefiningEquationForU1H}
	\end{align}
	and
	\begin{align}
	\frac{1}{h^2} \Big(D^2\tilde{W}&(\nabla_h u_{0,h}) \nabla_h u_{2,h}, \nabla_h \varphi\Big)_{L^2(\Omega)} = (h^2\partial_t^2 f|_{t=0} - u_{4,h})_{L^2(\Omega)}\label{GlobalDefiningEquationForU3H}\\
	& - \frac{1}{h^2}\Big(D^3\tilde{W}(\nabla_h u_{0,h})[\nabla_h u_{1,h}, \nabla_h u_{1,h}], \nabla_h\varphi\Big)_{L^2(\Omega)} - \frac{\gamma_h}{h^3} \Big(D^2\tilde{W}(\nabla_h u_{0,h})P,\nabla_h \varphi\Big)_{L^2(\Omega)}\notag
	\end{align}
	for all $\varphi\in\B$, where 
	\begin{equation*}
	\gamma_h(u_{0,h}) := \frac{1}{\mu(S)h^3} \Big(D\tilde{W}(\nabla_h u_{0,h}), P\Big)_{L^2(\Omega)}
	\end{equation*}
	and 
	\begin{equation*}
	P:= \begin{pmatrix}
	0 & 0 & 0\\
	0 & 0 & -1\\
	0 & 1 & 0
	\end{pmatrix}.
	\end{equation*}
	The solution satisfies
	\begin{gather}
	\bigg\|\bigg(\frac{1}{h}\e_h (u_{0,h}), \nabla \frac{1}{h}\e_h (u_{0,h}), \nabla_h^2 u_{0,h}\bigg)\bigg\|_{H^{1,1}(\Omega)} \leq Ch^2 \label{LocalInequalityForU0H}\\
	\max_{j=1,2} \bigg\|\bigg(\frac{1}{h}\e_h (u_{j,h}), \nabla \frac{1}{h}\e_h (u_{j,h}), \nabla_h^2  u_{j,h}\bigg)\bigg\|_{H^{2-j}(\Omega)}\leq Ch^2 \label{LocalInequalityForU1HandU2H}
	\end{gather}
	and $u_{k,h}\in \mathcal{B}$ for $k=0,1,2$. Moreover we have
	\begin{equation}
	\max_{j=0,1,2} \bigg\|\bigg(\frac{1}{h}\e_h(u_{j,h}) - \frac{1}{h}\e_h(\tilde{u}_{j,h})\bigg)\bigg\|_{L^2(\Omega)} \leq 
	\begin{cases}
	Ch^3 \quad &\text{ if } j=0,1,\\
	Ch^2 \quad &\text{ if } j=2
	\end{cases}
	\label{GlobalInequalityBoundForDifferenzeOfInitialValues}
	\end{equation}
	for all $h\in (0,h_0]$ and $C>0$ independent of $h$.
	\label{LemmaExistenceOfU1U2andInequality}
\end{lem}
\begin{proof}
	We can equivalently formulate \eqref{GlobalDefiningEquationForU0H}--\eqref{GlobalDefiningEquationForU3H} via
	\begin{equation}
	\frac{1}{h^2}\Big(D^2\tilde{W}(\nabla_h u_{0,h})\nabla_h u_{1,h}, \nabla_h \varphi\Big)_{L^2(\Omega)} = (h^2\partial_t f_h|_{t=0}, \varphi)_{L^2(\Omega)} - (u_{3,h}, \varphi)_{L^2(\Omega)}
	\label{LocalEquivalentFormulationForU2H}
	\end{equation}
	and
	\begin{align}
	\frac{1}{h^2}\Big(D^2\tilde{W}(\nabla_h u_{0,h})\nabla_h u_{2,h}&, \nabla_h \varphi\Big)_{L^2(\Omega)} = (h^2\partial_t^2 f_h|_{t=0} - u_{4,h}, \varphi)_{L^2(\Omega)}\notag\\
	&-\frac{1}{h^2} \Big(D^3\tilde{W}(\nabla_h u_{0,h}) [\nabla_h u_{1,h}, \nabla_h u_{1,h}], \nabla_h \varphi\Big)_{L^2(\Omega)} \notag\\
	&- \frac{\gamma_h(u_{0,h})}{h^3} \Big(D^2\tilde{W}(\nabla_h u_{0,h}) P,\nabla_h \varphi\Big)_{L^2(\Omega)}\label{LocalEquivalentFormulationForU3H}
	\end{align}
	for all $\varphi\in\mathcal{B}$, where $u_{0,h} = \mathcal{G}_{h,f}(u_{0,h})$ is the solution of \eqref{GlobalDefiningEquationForU0H} with $f = h^2 f^h - u_{2,h}$. Defining
	\begin{equation*}
	\mathcal{G}_{0,h}(u_{2,h}) := \mathcal{G}_{h,f}(u_{0,h})
	\end{equation*} 
	and deploying \eqref{InequalityBoundsForDifferenzOfSolutionsForNonlinearEquation} we obtain for $u_{2,h}$, $u'_{2,h} \in H^{1,1}(\Omega;\R^3)$
	\begin{equation}
	\max_{k=0,1}\Big\|\partial_{x_1}^k \big(\mathcal{G}_{0,h}(u_{2,h}) - \mathcal{G}_{0,h}(u'_{2,h})\big)\Big\|_{\mathcal{X}_h} \leq C_0 \|u_{2,h} - u'_{2,h}\|_{H^{1,1}(\Omega)}
	\label{LocalInequalityBoundednessOfU1byU2}
	\end{equation}
	if $\|u_{2,h}\|_{H^{1,1}(\Omega)} \leq \frac{1}{2}M_0 h$, $\|u'_{2,h}\|_{H^{1,1}(\Omega)} \leq \frac{1}{2}M_0 h$ and $h^2\|f^h\|_{H^{1,1}(\Omega)}\leq \frac{1}{2}M_0 h$. This can always be achieved if $h_0\in (0,1]$ is small enough and $u_{2,h}$, $u'_{2,h}$ are of order $h^2$.
	
	Using the definition of $L_h$ it follows that \eqref{LocalEquivalentFormulationForU2H}--\eqref{LocalEquivalentFormulationForU3H} are equivalent to
	\begin{equation*}
	\langle L_h u_{1,h}, \varphi\rangle_{\mathcal{B}', \mathcal{B}} = (h^2\partial_t f_h|_{t=0} - u_{3,h}, \varphi)_{L^2(\Omega)} - \frac{1}{h^2} \Big(DG(\nabla_h u_{0,h}) \nabla_h u_{1,h}, \nabla_h \varphi\Big)_{L^2(\Omega)}
	\end{equation*}
	and
	\begin{align*}
	\langle L_h &u_{2,h}, \varphi\rangle_{\mathcal{B}', \mathcal{B}} = (h^2\partial^2_t f_h|_{t=0} - u_{4,h}, \varphi)_{L^2(\Omega)} - \frac{1}{h^2} \Big(DG(\nabla_h u_{0,h}) \nabla_h u_{2,h}, \nabla_h \varphi\Big)_{L^2(\Omega)}\\
	& - \frac{1}{h^2} \Big(D^3\tilde{W}(\nabla_h u_{0,h})[\nabla_h u_{1,h}, \nabla_h u_{1,h}], \nabla_h \varphi \Big)_{L^2(\Omega)} - \frac{\gamma_h(u_{0,h})}{h^3} \Big(D^2\tilde{W}(\nabla_h u_{0,h})P,\nabla_h \varphi\Big)_{L^2(\Omega)}
	\end{align*}
	for all $\varphi\in\mathcal{B}$. Defining now the relevant function spaces by
	\begin{align*}
	\mathcal{D}_h &:= H^2_{per}(\Omega;\Rtimes)\times H^1_{per}(\Omega;\Rtimes),
	\quad \mathcal{Z}_h := H^1_{per}(\Omega;\R^3)\times L^2 (\Omega;\R^3)\times \mathcal{D}_h,\\
          \mathcal{W}_h &:= \mathcal{X}_h \times \Big(H^2_{per}(\Omega;\R^3) \cap \mathcal{B}\Big)
          	\end{align*}
	with the respective norms defined by
	\begin{align*}
	\|(F_1, F_2)\|_{\mathcal{D}_h} &:= \max_{i=1,2} \Big(h^2\|\nabla_h F_i\|_{H^{2-i}(\Omega)} + \max_{\sigma =0,\ldots, 3-i}\|\partial_{x_1}^\sigma F_i\|_{(L^2_h)'}\Big),\\
	\|(f_1,f_2, F_1, F_2)\|_{\mathcal{Z}_h} & := \max_{i=1,2} \Big(h^2\|(f_i, \nabla_h F_i)\|_{H^{2-i}(\Omega)} + \|f_i\|_{H^{0,2-i}(\Omega)} + \max_{\sigma =0,\ldots, 3-i}\|\partial_{x_1}^\sigma F_i\|_{(L^2_h)'}\Big),\\
	\|(g_1, g_2)\|_{\mathcal{W}_h} &:= \max_{i=1,2} \bigg\|\bigg(\frac{1}{h}\e_h(g_i), \nabla\frac{1}{h}\e_h(g_i),\nabla_h^2 g_i\bigg)\bigg\|_{H^{2-i}(\Omega)}.
                                                     	\end{align*}
	With this we define the linear operator $\mathcal{L}_h^{-1}\colon\mathcal{Z}_h \to\mathcal{W}_h$ by mapping $(f_1,f_2, F_1, F_2)$ to the solution $(w_1, w_2)$ of
	\begin{equation}
	\langle L_h w_i, \varphi\rangle_{\mathcal{B}',\mathcal{B}} = (f_i,\varphi)_{L^2(\Omega)} - (F_i, \nabla_h \varphi)_{L^2(\Omega)}
	\end{equation}
	for $i=1,2$. Then due to \eqref{LocalInequalityHigherRegularityRightHandSidefandF}, Theorem \ref{TheoremHigherRegularity} and \eqref{LocalInequalityAnisotropicRegularityBoundRightHandSidefandF} we obtain 
	\begin{equation}
	\|(w_1, w_2)\|_{\mathcal{W}_h}\leq C \|(f_1,f_2,F_1,F_2)\|_{\mathcal{Z}_h}.
	\label{LocalBoundForTheMultilinearPart}
	\end{equation}
	Hence $\mathcal{L}_h^{-1}$ is a bijective, linear and bounded operator. For the nonlinearity we define \begin{equation*}
	\mathcal{Q}_h\colon\mathcal{W}_h\to\mathcal{D}_h
	\end{equation*}
	via
	\begin{align*}
	\begin{pmatrix} u_1 \\ u_2 \end{pmatrix} \mapsto &
	\begin{pmatrix}
	-\frac{1}{h^2} DG(\nabla_h u_0) \nabla_h u_1\\ -\frac{1}{h^2} DG(\nabla_h u_0) \nabla_h u_2 - \frac{1}{h^2} D^3\tilde{W}(\nabla_h u_0) [\nabla_h u_1, \nabla_h u_1] - \frac{\gamma_h(u_{0,h})}{h^3} D^2\tilde{W}(\nabla_h u_{0,h})[P]
	\end{pmatrix} \\
	&\qquad =: \begin{pmatrix} \mathcal{Q}_{1,h}(u_1, u_2)\\\mathcal{Q}_{2,h}(u_1, u_2) \end{pmatrix},
	\end{align*}
	where $u_0 := \mathcal{G}_{h,f - u_2}(u_{0,h})$ for some fixed $f\in H^{1,1}_{per}(\Omega)$ with $\|f\|_{H^{1,1}(\Omega)} \leq M h^2$ and $\int_\Omega f dx = 0$ and $G$ is defined as in \eqref{GlobalDefinitionOfNonlinearityG}.
	
	We deduce the contraction properties of $\mathcal{Q}_h$ similar as in the proof of Lemma \ref{LemmaExistenceOfU0H}. For this we assume that $\|(u_1, u_2)\|_{\mathcal{W}_h}$, $\|(u'_1, u'_2)\|_{\mathcal{W}_h}\leq CM_2 h^2$. Starting with $\mathcal{Q}_{1,h}$ we obtain
	\begin{align*}
	\|\mathcal{Q}_{1,h}&(u_1, u_2) - \mathcal{Q}_{1,h}(u'_1, u'_2)\|_{(L^2_h)'}\\
	& = \frac{1}{h^2} \bigg\|\int_0^1 D^3\tilde{W}(\tau\nabla_h u_0)[\nabla_h u_0, \nabla_h u_1] d\tau - \int_0^1 D^3\tilde{W}(\tau\nabla_h u'_0)[\nabla_h u'_0, \nabla_h u'_1]d\tau\bigg\|_{(L^2_h)'}\\
	& \leq \frac{C}{h}\|\nabla_h (u_0 - u'_0)\|_{H^2_h(\Omega)} \|\nabla_h u_1\|_{L^2_h(\Omega)} + \frac{C}{h} \|\nabla_h u'_0\|_{H^2_h(\Omega)} \|\nabla_h (u_1 - u'_1)\|_{L^2_h(\Omega)}\\
	&\quad + CM_2 \bigg\|\frac{1}{h}\e_h(u_0 - u'_0)\bigg\|_{H^1_h(\Omega)}\\
	& \leq CM_2 \|u_2 - u'_2\|_{H^{1,1}(\Omega)} + CM_2 \bigg\|\frac{1}{h}\e_h(u_1 - u'_1)\bigg\|_{L^2(\Omega)} \leq CM_2 \|(u_1 - u'_1, u_2 - u'_2)\|_{\mathcal{W}_h},
	\end{align*}
	where we used \eqref{LocalInequalityBoundednessOfU1byU2}. Similarly one deduces that
	\begin{align*}
	\|\partial_{x_j} (\mathcal{Q}_{1,h}(u_1, u_2) - \mathcal{Q}_{1,h}(u'_1, u'_2))\|_{L^2(\Omega)} \leq CM_2 \|(u_1 - u'_1, u_2 - u'_2)\|_{\mathcal{W}_h}\\
	\|\partial_{x_k} \partial_{x_j} (\mathcal{Q}_{1,h}(u_1, u_2) - \mathcal{Q}_{1,h}(u'_1, u'_2))\|_{L^2(\Omega)} \leq CM_2 \|(u_1 - u'_1, u_2 - u'_2)\|_{\mathcal{W}_h}
	\end{align*}
	for $j,k=1,2,3$. Analogously we deduce for $\mathcal{Q}_{2,h}$ 
	\begin{align*}
	\|\mathcal{Q}_{2,h}&(u_1, u_2) - \mathcal{Q}_{2,h}(u'_1, u'_2)\|_{(L^2_h)'}\\
	& \leq \frac{1}{h^2}\bigg\|\int_0^1 D^3\tilde{W}(\tau\nabla_h u_0)[\nabla_h u_0, \nabla_h u_2] - D^3\tilde{W}(\tau\nabla_h u'_0)[\nabla_h u'_0, \nabla_h u'_2]d\tau\bigg\|_{(L^2_h)'}\\
	&\quad + \frac{|\gamma_h(u_{0,h})|}{h^3}\bigg\|\int_0^1 D^3\tilde{W}(\tau\nabla_h u_0)[\nabla_h u_0, P] - D^3\tilde{W}(\tau\nabla_h u'_0)[\nabla_h u'_0, P]d\tau\bigg\|_{(L^2_h)'}\\
	& \quad + \frac{|\gamma_h(u_{0,h}) - \gamma_h (u_{0,h}')|}{h^3} \bigg\|\int_0^1 D^3\tilde{W}(\tau\nabla_h u'_0)[\nabla_h u'_0, P]d\tau\bigg\|_{(L^2_h)'}\\
	&\quad + \frac{1}{h^2} \bigg\| D^3\tilde{W}(\nabla_h u_0)[\nabla_h u_1, \nabla_h u_1] - D^3\tilde{W}(\nabla_h u'_0)[\nabla_h u'_1, \nabla_h u'_1]\bigg\|_{(L^2_h)'}\\
	& \leq  \frac{C}{h}\|\nabla_h (u_0 - u'_0)\|_{H^2_h(\Omega)} \|\nabla_h u_2\|_{L^2_h(\Omega)} + \frac{C}{h} \|\nabla_h u'_0\|_{H^2_h(\Omega)} \|\nabla_h (u_2 - u'_2)\|_{L^2_h(\Omega)}\\
	& \quad +  \frac{C}{h}\|\nabla_h (u_1 - u'_1)\|_{H^2_h(\Omega)} \|\nabla_h u_1\|_{L^2_h(\Omega)} + \frac{C}{h} \|\nabla_h u'_1\|_{H^2_h(\Omega)} \|\nabla_h (u_1 - u'_1)\|_{L^2_h(\Omega)}\\
	& \quad + \frac{C}{h^2} \|\nabla_h(u_0 - u'_0)\|_{H^2_h(\Omega)} \|\nabla_h u_1\|_{H^1_h(\Omega)} \|\nabla_h u'_1 \|_{H^1_h(\Omega)}\\
	& \quad + \frac{C}{h^2} \|\nabla_h(u_0 - u'_0)\|_{H^2_h(\Omega)} \|\nabla_h u'_0\|_{H^2_h(\Omega)} \|\nabla_h u_2 \|_{L^2_h(\Omega)}\\
	& \quad + \frac{|\gamma_h(u_{0,h})|}{h^2} \|\nabla_h(u_0 - u'_0)\|_{H^2_h(\Omega)} + |\gamma_h(u_{0,h}) - \gamma_h(u_{0,h}')|\\
	& \leq CM_2 \|(u_1 - u'_1, u_2 - u'_2)\|_{\mathcal{W}_h}, 
	\end{align*}
	where we used again Corollary \ref{Corollary_L1BoundsD3TildeW}, $|P|_h = |P|$, $|\gamma_h(u_{0,h})| \leq Ch^2$ and
	\begin{align*}
	|\gamma_h(u_{0,h}) - \gamma_h(u_{0,h}')| &\leq \frac{1}{h^3} \bigg\|\int_0^1 (1-\tau) \Big(D^3\tilde{W}(\nabla_h u_{0,h})[\nabla_h u_{0,h}, \nabla_h u_{0,h}]\\
	&\hspace{4cm} - D^3\tilde{W}(\nabla_h u'_{0,h})[\nabla_h u'_{0,h}, \nabla_h u'_{0,h}]\Big)\bigg\|_{(L^2_h)'}\\
	& \leq CM_2 \|\nabla_h(u_0 - u'_0)\|_{H^2_h(\Omega)} \leq CM_2 \|(u_1 - u'_1, u_2 - u'_2)\|_{\mathcal{W}_h}.
	\end{align*}
	Finally from 
	\begin{align*}
	\partial_{x_j} \mathcal{Q}_{2,h}(u_1, u_2) &= \frac{1}{h^2} \int_0^1 D^3\tilde{W}(\tau\nabla_h u_0)[\nabla_h u_0, \nabla_h\partial_{x_j} u_2] d\tau + \frac{1}{h^2} D^3\tilde{W}(\nabla_h u_0) [\nabla_h \partial_{x_j} u_0, \partial_h u_2]\\
	&\quad + \frac{2}{h^2} D^3\tilde{W}(\nabla_h u_0) [\nabla_h \partial_{x_j} u_1, \nabla_h u_1] + \frac{1}{h^2} D^4\tilde{W}(\nabla_h u_0) [\nabla_h \partial_{x_j} u_0, \nabla_h u_1, \nabla_h u_1]\\
	&\quad - \frac{\gamma_h}{h^3} D^3\tilde{W}(\nabla_h u_{0,h})[\nabla_h\partial_{x_j} u_0, P]
	\end{align*}
	it follows
	\begin{equation*}
	\|\partial_{x_j} (\mathcal{Q}_{2,h}(u_1, u_2) - \mathcal{Q}_{2,h}(u'_1, u'_2))\|_{(L^2_h)'} \leq CM_2 \|(u_1 - u'_1, u_2 - u'_2)\|_{\mathcal{W}_h}.
	\end{equation*}
	Choosing now $M_2 \in (0,1]$ small enough we obtain that
	\begin{equation*}
	\mathcal{F}_{h, f_0,f_1,f_2}\colon\overline{B_{CM_2 h^2}(0)} \subset\mathcal{X}_h\times \mathcal{W}_h \to \mathcal{X}_h\times \mathcal{W}_h
	\end{equation*}
	defined by
	\begin{equation*}
	\begin{pmatrix} u_0 \\ u_1 \\ u_2 \end{pmatrix} \mapsto 
	\begin{pmatrix}
	\mathcal{G}_{h, f_0 - u_{2}} (u_0)\\
	\mathcal{L}_h^{-1}\left(\begin{pmatrix}	f_1 \\ f_2 \end{pmatrix}, \mathcal{Q}_{h} (u_1, u_2)\right)
	\end{pmatrix}
	\end{equation*}
	is a $\frac{1}{2}$-contraction, where $f_0 := h^2 f^h|_{t=0}$,  $f_1 := h^2\partial_t f^h|_{t=0} - u_{3,h}$ and $f_2 := h^2\partial^2_t f^h|_{t=0} - u_{4,h}$. We can use an analogous argument as in Lemma \ref{LemmaExistenceOfU0H}. First it holds, due to \eqref{lokalAssumptionOfSmallnessForInitialData} and \eqref{lokalAssumptionOfSmallnessForRightHandSideG}, for $M>0$ sufficiently small
	\begin{equation*}
	\|\mathcal{F}_{h, f_0,f_1,f_2}(0)\|_{\mathcal{X}_h\times\mathcal{W}_h} \leq \tilde{C} M h^2 \leq \frac{C M_2 h^2}{2}
	\end{equation*}
	and with the $\frac{1}{2}$-contraction property we obtain the self mapping of $\mathcal{F}_{h, f_0,f_1,f_2}$. Moreover due to the norm on $\mathcal{X}_h$ and $\mathcal{W}_h$ we obtain \eqref{LocalInequalityForU0H} and \eqref{LocalInequalityForU1HandU2H}, respectively.
	
	Finally, the construction of $\tilde{u}_h$ implies that $\tilde{u}_{j,h}$ satisfies
	\begin{align*}
	\frac{1}{h^2}\Big(D^2\tilde{W}(0) \nabla_h \tilde{u}_{j,h}, \nabla_h\varphi\Big)_{L^2(\Omega)} & = \Big(h^2 \partial_t^j f^h|_{t=0} - \tilde{u}_{2+j,h}, \varphi\Big)_{L^2(\Omega)} + (\partial_t^j r_h, \varphi)_{L^2(\Omega)}\\
	& \qquad - \frac{1}{h^2} \int_0^L \Big(\tr{\partial S}(\partial_t^j r_{N,h}(x_1, \cdot)), \tr{\partial S}(\varphi(x_1, \cdot))\Big)_{L^2(\partial S)} dx_1
	\end{align*}
	for $j=0,1,2$ and all $\varphi\in\mathcal{B}$. This implies with \eqref{GlobalDefiningEquationForU0H}--\eqref{GlobalDefiningEquationForU3H}
	\begin{align*}
	\frac{1}{h^2} \big(\e_h(u_{1,h} &- \tilde{u}_{1,h}), \e_h(\varphi)\big)_{L^2(\Omega)} = -\frac{1}{h^2}\Big((D^2\tilde{W}(\nabla_h u_{0,h})- D^2\tilde{W}(0)) \nabla_h u_{1,h}, \nabla_h \varphi\Big)_{L^2(\Omega)}\\
	& + (r_{1,h}, \varphi)_{L^2(\Omega)} - \frac{1}{h^2} \int_0^L \Big(\tr{\partial S}(\partial_t r_{N,h}(x_1, \cdot)), \tr{\partial S}(\varphi(x_1, \cdot))\Big)_{L^2(\partial S)} dx_1\\
	\frac{1}{h^2}\big(\e_h(u_{2,h} &- \tilde{u}_{2,h}), \e_h(\varphi)\big)_{L^2(\Omega)} = -\frac{1}{h^2}\Big((D^2\tilde{W}(\nabla_h u_{0,h})- D^2\tilde{W}(0)) \nabla_h u_{2,h}, \nabla_h \varphi\Big)_{L^2(\Omega)}\\
	+ & (r_{2,h}, \varphi)_{L^2(\Omega)} - \frac{1}{h^2} \int_0^L \Big(\tr{\partial S}(\partial_t^2 r_{N,h}(x_1, \cdot)), \tr{\partial S}(\varphi(x_1, \cdot))\Big)_{L^2(\partial S)} dx_1\\
	- & \frac{1}{h^2} \Big(D^3\tilde{W}(\nabla_h u_{0,h}) [\nabla_h u_{1,h}, \nabla_h u_{1,h}], \nabla_h \varphi\Big)_{L^2(\Omega)} - \frac{\gamma_h}{h^3} \Big(D^2\tilde{W}(\nabla_h u_{0,h}) P,\nabla_h \varphi\Big)_{L^2(\Omega)}
	\end{align*}
	for all $\varphi\in\mathcal{B}$, where we defined 
	\[r_{j,h} := u_{2+j,h} - \tilde{u}_{2+j,h} - \partial_t^j r_h.\]
	With this it follows $\max_{j=1,2} \|r_{j,h}\|_{C^0(0,T;L^2(\Omega))}\leq Ch^3$ because of the definition of $u_{2+j,h}$ and the bound on $\partial_t r_h$. Additionally we have due to Lemma \ref{Lemma_DecompD3TildeW} and Corollary \ref{Corollary_L1BoundsD3TildeW}, the bounds on $(u_{0,h}, u_{1,h}, u_{2,h})$ and $\varphi\in\mathcal{B}$ 
	\begin{align*}
	\bigg| \frac{1}{h^2}\Big(&(D^2\tilde{W}(\nabla_h u_{0,h})- D^2\tilde{W}(0)) \nabla_h u_{j,h}, \nabla_h \varphi\Big)_{L^2(\Omega)} \bigg|\\
	&= \bigg|\frac{1}{h^2} \int_0^1 \Big(D^3\tilde{W}(\tau\nabla_h u_{0,h}) [\nabla_h u_{0,h}, \nabla_h u_{j,h}], \nabla_h \varphi\Big)_{L^2(\Omega)} d\tau\bigg| \leq Ch^3 \bigg\|\frac{1}{h}\e_h(\varphi)\bigg\|_{L^2(\Omega)}
	\end{align*}
	as well as 
	\begin{equation*}
	\bigg|\frac{1}{h^2} \Big(D^3\tilde{W}(\nabla_h u_{0,h}) [\nabla_h u_{1,h}, \nabla_h u_{1,h}], \nabla_h \varphi\Big)_{L^2(\Omega)}\bigg| \leq Ch^3 \bigg\|\frac{1}{h}\e_h(\varphi)\bigg\|_{L^2(\Omega)}
	\end{equation*}
	and
	\begin{equation*}
	\bigg|\frac{\gamma_h}{h^3} \Big(D^2\tilde{W}(\nabla_h u_{0,h}) P,\nabla_h \varphi\Big)_{L^2(\Omega)}\bigg| \leq Ch^2 \bigg\|\frac{1}{h}\e_h(\varphi)\bigg\|_{L^2(\Omega)}.
	\end{equation*}
	Regarding the boundary terms we use that $\tr{\partial S}\colon H^1(S)\to H^{\frac{1}{2}}(\partial S)$ is linear and bounded. Hence for $j=0,1,2$
	\begin{align}
	\bigg|\frac{1}{h^2} \int_0^L \Big(& \tr{\partial S}(\partial_t^j r_{N,h}(x_1, \cdot)), \tr{\partial S}(\varphi(x_1, \cdot))\Big)_{L^2(\partial S)} dx_1\bigg|\notag \\
	& \leq \frac{1}{h^2} \|\partial_t^j r_{N,h}\|_{L^2(0,L;H^1(S))} \|\varphi\|_{L^2(0,L;H^1(S))} \leq Ch^3 \bigg\|\frac{1}{h}\e_h(\varphi)\bigg\|_{L^2(\Omega)},
	\label{LocalInequalityBoundForBoundaryData}
	\end{align}
	where we used that $\|r_{N,h}\|_{C^2([0,T]; H^1(\Omega))} \leq Ch^5$ and the Poincaré and Korn inequality for $\varphi$. Choosing $\varphi = u_{j,h} - \tilde{u}_{j,h}$ it follows with an absorption argument
	\begin{equation*}
	\max_{j=1,2} \bigg\|\frac{1}{h}\e_h(u_{j,h}) -\frac{1}{h}\e_h(\tilde{u}_{j,h})\bigg\|_{L^2(\Omega)} \leq \begin{cases}
	Ch^3, \quad &\text{ if } j=1,\\
	Ch^2, \quad &\text{ if } j=2.
	\end{cases}
	\end{equation*}
	Now, for $u_{0,h}-\tilde{u}_{0,h}$ it holds
	\begin{align*}
	\frac{1}{h^2} (\e_h(u_{0,h}-\tilde{u}_{0,h}),& \e_h(\varphi))_{L^2(\Omega)} = -\frac{1}{h^2} (G(\nabla_h u_{0,h}), \nabla_h \varphi)_{L^2(\Omega)}\\
	& + (r_{0,h}, \varphi)_{L^2(\Omega)} - \frac{1}{h^2} \int_0^L \Big(\tr{\partial S}(r_{N,h}(x_1, \cdot)), \tr{\partial S}(\varphi(x_1, \cdot))\Big)_{L^2(\partial S)} dx_1.
	\end{align*}
	The definition of $G$ implies now 
	\begin{align*}
	\bigg|\frac{1}{h^2} (G(\nabla_h u_{0,h}), \nabla_h \varphi)_{L^2(\Omega)}\bigg| &= \bigg|\frac{1}{h^2} \int_0^1 (1-\tau) \Big(D^3\tilde{W}(\tau\nabla_h u_{0,h})[\nabla_h u_{0,h}, \nabla_h u_{0,h}], \nabla_h\varphi\Big)_{L^2(\Omega)} d\tau\bigg|\\
	& \leq Ch^3 \bigg\|\frac{1}{h}\e_h(\varphi)\bigg\|_{L^2(\Omega)}
	\end{align*}
	because of the bounds for $u_{0,h}$ and Corollary \ref{Corollary_L1BoundsD3TildeW}. Using \eqref{LocalInequalityBoundForBoundaryData} it follows 
	\begin{equation*}
	\bigg\|\frac{1}{h}\e_h(u_{0,h}) -\frac{1}{h}\e_h(\tilde{u}_{0,h})\bigg\|_{L^2(\Omega)} \leq Ch^3\qedhere
	\end{equation*}
\end{proof}

\subsection{Main Result}
\label{subsection::MainResult}
\begin{theorem}\label{thm:Main}
	Let $f_h$, $\tilde{v}_0$, $\tilde{v}_1$, $\tilde{u}_{j,h}$, $j=0,1,2$ and $\tilde{u}_h$ be given as above. Then there exists some $h_0\in (0,1]$ such that for $h\in (0,h_0]$ there are initial values $(u_{0,h}, u_{1,h})$ satisfying \eqref{AssumptionOnInitialDataU1}--\eqref{AssumptionOnInitialDataU3} and such that
	\begin{equation*}
	\max_{j=0,1} \bigg\|\frac{1}{h} \e_h(u_{j,h}) - \frac{1}{h} \e_h(\tilde{u}_{j,h})\bigg\|_{L^2(\Omega)} \leq Ch^3.
	\end{equation*}
	Moreover, if $u_h$ solves \eqref{NLS1}--\eqref{NLS4}, then
	\begin{equation*}
	\bigg\|\bigg((u_h - \tilde{u}_h), \frac{1}{h} \int_0^t \e_h\big(u_h(\tau) - \tilde{u}_h(\tau)\big) d\tau\bigg)\bigg\|_{L^\infty(0,L; L^2(\Omega))} \leq Ch^3\quad \text{for all }0<h\leq h_0.
	\end{equation*}
\end{theorem}
\begin{proof}
	Given $(u_{3,h}, u_{4,h})$ we construct $(u_{0,h}, u_{1,h}, u_{2,h})$ such that \eqref{AssumptionOnInitialDataU1}--\eqref{AssumptionOnInitialDataU3} holds. First we note that $\|u_{4,h}\|_{L^2(\Omega)}$ is of order $h^2$ as $\partial_{x_1}^l v^{4}$ is bounded in $L^2(0,L)$ for $l=0,1$. Moreover we have
	\begin{equation*}
	\int_\Omega u_{2+j, h} dx = 0
	\end{equation*}
	for $j=1,2$ and 
	\begin{equation*}
	\int_\Omega u_{3,h}\cdot x^\perp dx = 0.
	\end{equation*}
	Using the structure of $u_{3,h}$ we obtain
	\begin{equation*}
	\frac{1}{h}\e_h(u_{3,h}) = h^2
	\begin{pmatrix}
	-x_2 \partial_{x_1}^2 v_{2}^3 - x_3 \partial_{x_1}^2 v_{3}^3 & 0  & 0\\
	0 & 0 & 0 \\
	0 & 0 & 0
	\end{pmatrix}.
	\end{equation*}
	Altogether we obtain that $u_{3,h}$ and $u_{4,h}$ satisfy \eqref{AssumptionOnInitialDataU1}--\eqref{AssumptionOnInitialDataU3}, the necessary conditions for the large times existence result in the appendix. The assumptions on $g$ and the structure of $f_h$ imply that \eqref{AssumptionsOnNonlinearf} and \eqref{AssumptionsOnRotationOfNonlinearf} are fulfilled. Applying Lemma \ref{LemmaExistenceOfU0H} and \ref{LemmaExistenceOfU1U2andInequality} we obtain for $h_0$ sufficiently small, the existence of $(u_{0,h}, u_{1,h}, \bar{u}_{2,h})$ such that
	\begin{align*}
	\frac{1}{h^2} \Big(D\tilde{W}(\nabla_h u_{0,h}), \nabla_h \varphi\Big)_{L^2(\Omega)} &= (h^2 f_h|_{t=0}, \varphi)_{L^2(\Omega)} - (\bar{u}_{2,h}, \varphi)_{L^2(\Omega)}\\
	\frac{1}{h^2}\Big(D^2\tilde{W}(\nabla_h u_{0,h})\nabla_h u_{1,h}, \nabla_h \varphi\Big)_{L^2(\Omega)} &= (h^2 \partial_t f_h|_{t=0}, \varphi)_{L^2(\Omega)} - (u_{3,h}, \varphi)_{L^2(\Omega)}
	\end{align*}
	and 
	\begin{align*}
	\frac{1}{h^2}\Big(D^2\tilde{W}(&\nabla_h u_{0,h})\nabla_h \bar{u}_{2,h}, \nabla_h \varphi\Big)_{L^2(\Omega)} = (h^2 \partial_t^2 f_h|_{t=0} - u_{4,h}, \varphi)_{L^2(\Omega)}\\
	-\frac{1}{h^2} &\Big(D^3\tilde{W}(\nabla_h u_{0,h}) [\nabla_h u_{1,h}, \nabla_h u_{1,h}], \nabla_h \varphi\Big)_{L^2(\Omega)} - \frac{\gamma_h}{h^3} \Big(D^2\tilde{W}(\nabla_h u_{0,h}) P,\nabla_h \varphi\Big)_{L^2(\Omega)}
	\end{align*}
	for all $\varphi\in\mathcal{B}$. We use the ansatz $u_{2,h} = \bar{u}_{2,h} + \gamma_2^h x^\perp$ and $\bar{u}_{2+j,h} = u_{2+j,h} + \gamma_{2+j}^h x^\perp$ for $j=1,2$. Choosing
	\begin{equation*}
	\gamma_2^h := -\frac{1}{\mu(S) h^2} \Big(D\tilde{W}(\nabla_h u_{0,h}), \nabla_h x^\perp\Big)_{L^2}
	\end{equation*}
	it follows
	\begin{equation}
	\frac{1}{h^2} \Big(D\tilde{W}(\nabla_h u_{0,h}), \nabla_h \varphi\Big)_{L^2(\Omega)} = (h^2 f_h|_{t=0}, \varphi)_{L^2(\Omega)} - (u_{2,h}, \varphi)_{L^2(\Omega)}
	\end{equation}
	for all $\varphi\in H^1_{per}(\Omega;\R^3)$. Moreover, for 
	\begin{equation*}
	\gamma_3^h := \frac{1}{\mu(S)h^2} \Big(D^2\tilde{W}(\nabla_h u_{0,h})\nabla_h u_{1,h}, \nabla_h x^\perp\Big)_{L^2}
	\end{equation*}
	we deduce
	\begin{equation}
	\frac{1}{h^2}\Big(D^2\tilde{W}(\nabla_h u_{0,h})\nabla_h u_{1,h}, \nabla_h \varphi\Big)_{L^2(\Omega)} = (h^2 \partial_t f_h|_{t=0}, \varphi)_{L^2(\Omega)} - (\bar{u}_{3,h}, \varphi)_{L^2(\Omega)}
	\end{equation}
	for all $\varphi\in H^1_{per}(\Omega;\R^3)$. Then it holds $|\gamma_2^h|\leq C h^2$ as
	\begin{equation*}
	D\tilde{W}(\nabla_h u_{0,h}) = D^2\tilde{W}(0)[\nabla_h u_{0,h}] + \int_0^1 (1-\tau) D^3\tilde{W}(\tau \nabla_h u_{0,h})[\nabla_h u_{0,h}, \nabla_h u_{0,h}] d\tau
	\end{equation*}
	and $|\gamma_3^h| \leq C h^2$ with a similar calculation. Lastly, we need to find $\gamma_4^h$ such that 
	\begin{align}
	\frac{1}{h^2}\Big(D^2\tilde{W}(\nabla_h u_{0,h})\nabla_h u_{2,h}, \nabla_h \varphi\Big)_{L^2(\Omega)} = (h^2 \partial_t^2 f_h|_{t=0} - \bar{u}_{4,h}, \varphi)_{L^2(\Omega)}\\
	-\frac{1}{h^2} \Big(D^3\tilde{W}(\nabla_h u_{0,h}) [\nabla_h u_{1,h}, \nabla_h u_{1,h}], \nabla_h \varphi\Big)_{L^2(\Omega)}\notag
	\end{align}
	for all $\varphi\in H^1_{per}(\Omega;\R^3)$. Therefore we choose
	\begin{align*}
	\gamma_4^h &:= -\frac{1}{h^2}\Big(D^2\tilde{W}(\nabla_h u_{0,h})\nabla_h u_{2,h}, \nabla_h x^\perp\Big)_{L^2} -\frac{\gamma_2^h}{h^2}\Big(D^2\tilde{W}(\nabla_h u_{0,h})\nabla_h x^\perp, \nabla_h x^\perp\Big)_{L^2}\\
	& \qquad + \frac{1}{h^2} \Big(D^3\tilde{W}(\nabla_h u_{0,h})[\nabla_h u_{1,h}, \nabla_h u_{1,h}], \nabla_h x^\perp\Big)_{L^2}.
	\end{align*}		
	The first and last term can be bounded easily, using Corollary \ref{Corollary_L1BoundsD3TildeW}
	\begin{align*}
	\bigg|\frac{1}{h^2}\Big(D^2\tilde{W}(\nabla_h u_{0,h}) \nabla_h \bar{u}_{0,h}, \nabla_h x^\perp\Big)_{L^2} \bigg| &= \bigg|\frac{1}{h^3} \int_0^1 \Big(D^3\tilde{W}(\nabla_h u_{0,h})[\nabla_h u_{0,h}, \nabla_h\bar{u}_{2,h}], P\Big)_{L^2}\bigg| \\
	&\leq \frac{C}{h^2} \|\nabla_h u_{0,h}\|_{H^1_h} \|\nabla_h \bar{u}_{2,h}\|_{H^1_h} \leq Ch^2
	\end{align*}
	and 
	\begin{equation*}
	\bigg|\frac{1}{h^2} \Big(D^3\tilde{W}(\nabla_h u_{2,h})[\nabla_h u_{1,h}, \nabla_h u_{1,h}], \nabla_h x^\perp\Big)_{L^2} \bigg| \leq Ch^2.
	\end{equation*}
	For the second part of $\gamma_4^h$ we use the following equality 
	\begin{align*}
	\Big(D^2\tilde{W}(\nabla_h u_{0,h}) P, P\Big)_{L^2(\Omega)} &= \Big(D^3\tilde{W}(0) [\nabla_h u_{0,h}, P], P\Big)_{L^2(\Omega)} \\
	& \qquad + \int_0^1 (1-\tau) \Big(D^4\tilde{W}(\tau \nabla_h u_{0,h})[\nabla_h u_{0,h}, \nabla_h u_{0,h}, P], P\Big)_{L^2(\Omega)} d\tau
	\end{align*}
	where 
	\begin{equation}
	\bigg|\Big(D^4\tilde{W}(\tau \nabla_h u_{0,h})[\nabla_h u_{0,h}, \nabla_h u_{0,h}, P], P\Big)_{L^2(\Omega)}\bigg| \leq Ch^4 \quad\text{ for all } \tau\in[0,1]
	\end{equation}
	as $\|u_{0,h}\|_{H^1_h(\Omega)} \leq Ch^2$ and $|P|_h = |P|$, because $P\in\Rtimes_{skew}$. Furthermore, we obtain with \eqref{Global_ThirdOrderD3WEquality}
	\begin{align*}
		\Big(D^3\tilde{W}(0) [\nabla_h u_{0,h}, P], P\Big)_{L^2(\Omega)} = h \bigg(D^3\tilde{W}(0) &\bigg[\frac{1}{h}\e_h(u_{0,h}) - \frac{1}{h}\e_h(\tilde{u}_{0,h}), P\bigg], P\bigg)_{L^2(\Omega)}\\
		&+ \Big(D^3\tilde{W}(0) [\nabla_h \tilde{u}_{0,h}, P], P\Big)_{L^2(\Omega)}.
	\end{align*}
	Utilizing the inequality for the initial values \eqref{GlobalInequalityBoundForDifferenzeOfInitialValues}, we deduce 
	\begin{equation*}
	\bigg| h \bigg(D^3\tilde{W}(0) \bigg[\frac{1}{h}\e_h(u_{0,h}) - \frac{1}{h}\e_h(\tilde{u}_{0,h}), P\bigg], P\bigg)_{L^2}\bigg| \leq C h^4.
	\end{equation*}
	Lastly due to the symmetry properties of $D^3\tilde{W}$, the structure of $\nabla_h \tilde{u}_{0,h}$ and \eqref{Global_ThirdOrderD3WEquality} it follows
	\begin{equation*}
	\bigg|\Big(D^3\tilde{W}(0) [\operatorname{sym}(\nabla_h \tilde{u}_{0,h}), P], P\Big)_{L^2}\bigg| = \bigg|\Big(D^3\tilde{W}(0) [h^3 Q, P], P\Big)_{L^2} + \Big(D^3\tilde{W}(0) [R, P], P\Big)_{L^2}\bigg|
	\end{equation*}
	where
	\begin{align*}
	Q &= \begin{pmatrix}
	-x_2 \partial_{x_1}^2 v_2 - x_3 \partial_{x_1}^2 v_3 & 0 & 0\\
	0 & 0 & 0 \\
	0 & 0 & 0
	\end{pmatrix}\\
	R & = h^4 \operatorname{sym}\begin{pmatrix}
	0 & \partial_{x_2} a_2 \partial_{x_1}^3 v_2 + \partial_{x_2} a_3 \partial_{x_1}^3 v_3 & \partial_{x_3} a_2 \partial_{x_1}^3 v_2 + \partial_{x_3} a_3\partial_{x_1}^3 v_3 \\
	0&0&0\\
	0&0&0\end{pmatrix} \\
	& \quad +h^5 \operatorname{sym}\begin{pmatrix}
	a_2 \partial_{x_1}^4 v_2 + a_3 \partial_{x_1}^4 v_3 & 0 &0\\
	0& \partial_{x_2} b_2 \partial_{x_1}^4 v_2 + \partial_{x_2} c_3 \partial_{x_1}^4 v_3  & \partial_{x_3} b_2 \partial_{x_1}^4 v_2 +\partial_{x_3} c_3 \partial_{x_1}^4 v_3\\
	0 & \partial_{x_2} b_3 \partial_{x_1}^4 v_3 +\partial_{x_2} c_2 \partial_{x_1}^4 v_2 & \partial_{x_3} b_3 \partial_{x_1}^4 v_3 + \partial_{x_3} c_2 \partial_{x_1}^4 v_2
	\end{pmatrix}\\
	& \quad + h^6 \operatorname{sym}\begin{pmatrix}
	0 & 0 & 0 \\
	b_2 \partial_{x_1}^5 v_2 + c_3 \partial_{x_1}^5 v_3 & 0 & 0 \\
	b_3 \partial_{x_1}^5 v_3 + c_2 \partial_{x_1}^5 v_2  & 0 & 0
	\end{pmatrix}.
	\end{align*}
	Due to the structure of $Q$ and $P = \nabla x^\perp$ it follows 
	\begin{equation*}
	D^3\tilde{W}(0)[Q,P,P] = \Big((Q^T- Q)^T \operatorname{sym}(P) + (P^T - P)^T \operatorname{sym}(Q)\Big) : P = 0.
	\end{equation*}
	Hence, with $R = O(h^4)$ we obtain
	\begin{equation*}
	\bigg|\Big(D^3\tilde{W}(0) [\operatorname{sym}(\nabla_h \tilde{u}_{0,h}), P], P\Big)_{L^2}\bigg| \leq Ch^4.
	\end{equation*}
	Thus, altogether, it follows with $|\gamma_2^h| \leq Ch^2$
	\begin{equation*}
	\bigg|\frac{\gamma_2^h}{h^2} \Big(D^2\tilde{W}(\nabla_h u_{0,h})\nabla_h x^\perp, \nabla_h x^\perp\Big)_{L^2}\bigg| \leq Ch^2.
	\end{equation*}
	We obtain for $h_0$ sufficiently small, the existence of $(u_{0,h}, u_{1,h}, u_{2,h})$ such that \eqref{AssumptionOnInitialDataU1}--\eqref{AssumptionOnInitialDataU3} are satisfied and 
	\begin{equation*}
	\max_{j=0,1} \bigg\|\bigg(\frac{1}{h}\e_h(u_{j,h}) - \frac{1}{h}\e_h(\tilde{u}_{j,h})\bigg)\bigg\|_{L^2(\Omega)} \leq Ch^3
	\end{equation*}
	holds.
	
	Due to Theorem \ref{MainTheorem} there exists a solution $u_h\in \bigcap_{k=0}^4 C^k([0,T]; H^{4-k}_{per}(\Omega;\R^3))$ of \eqref{NLS1}--\eqref{NLS4}. Thus $w_h := u_h - \tilde{u}_h$ solves the system
	\begin{align*}
	-\big(\partial_t w_h, \partial_t\varphi&\big)_{L^2(Q_T)} + \frac{1}{h^2} \big(D^2\tilde{W}(0) \nabla_h w_h, \nabla_h \varphi\big)_{L^2(Q_T)} -  (w_1, \varphi|_{t=0})_{L^2(\Omega)}\\
	&= \frac{1}{h^2}\int_0^1 \big((D^2\tilde{W}(\tau \nabla_h u_h) - D^2\tilde{W}(0))\nabla_h\tilde{u}_h, \nabla_h \varphi\big)_{L^2(Q_T)}d\tau  - \big(r_h, \varphi\big)_{L^2(Q_T)}\\
	&\qquad - \frac{1}{h^2} (\operatorname{tr}_{\partial\Omega}(r_{N,h}), \operatorname{tr}_{\partial\Omega}(\varphi))_{L^2(0,T;L^2(\partial\Omega))},\\
	&w_h \text{ is } L\text{-periodic in } x_1\text{-direction},\\
	&w_h|_{t=0} = w_{0,h}
	\end{align*}
	for all $\varphi\in C^1([0,T]; H^1_{per, (0)}(\Omega;\R^3))$ with $\varphi|_{t=T} = 0$ and with $w_{j,h} := u_{j, h} - \tilde{u}_{j,h}$, $j=0,1$. Hence with \eqref{InequalitieLinearisedWeakForm} we obtain an upper bound for $w$. For this we use that, due to the structure of $r_h$ and $r_{N,h}$, it follows
	\begin{gather*}
	\frac{1}{h^2} \|r_{N,h}\|_{L^1(0,T;H^1)} \leq Ch^3,\qquad 
	\|r_h\|_{L^1(0,T;L^2)} \leq Ch^3
	\end{gather*}	
	as $a$, $b$, $c$ and $v$ are sufficiently regular. Moreover, using \eqref{GlobalInequalityBoundForDifferenzeOfInitialValues}
	\begin{equation*}
	\|w_{k,h}\|_{L^2(\Omega)} \leq 	\max_{j=0,1} \bigg\|\bigg(\frac{1}{h}\e_h(u_{j,h}) - \frac{1}{h}\e_h(\tilde{u}_{j,h})\bigg)\bigg\|_{L^2(\Omega)} \leq Ch^3
	\end{equation*}
	for $k=0,1$, where we used Poincaré's and Korn's inequality, as well as the fact that $w_{k,h}\in \mathcal{B}$ holds for $k=0,1$.
	With the fundamental theorem of calculus and Corollary \ref{Corollary_L1BoundsD3TildeW} we deduce 
	\begin{align*}
	\sup_{\varphi\in X_h, \|\varphi\|_{X_h} = 1} &\bigg|\frac{1}{h^2}\int_0^1\Big((D^2\tilde{W}(\tau \nabla_h u_h) - D^2\tilde{W}(0)) \nabla_h \partial_t \tilde{u}_h, \nabla_h \varphi\Big)_{L^2(\Omega)}d\tau \bigg|\\
	& \leq \sup_{\varphi\in X_h, \|\varphi\|_{X_h} = 1} \bigg|\frac{1}{h^2}\int_0^1 \int_0^1 \Big(D^3\tilde{W}(s \tau\nabla_h u_h) [\nabla_h u_h, \nabla_h \partial_t \tilde{u}_h], \nabla_h \varphi\Big)_{L^2(\Omega)}ds d\tau\bigg|\\
	&\leq \sup_{\varphi\in X_h, \|\varphi\|_{X_h} = 1} \frac{C}{h}\|\nabla_h u_h\|_{H^2_h(\Omega)}\|\nabla_h \partial_t \tilde{u}_h \|_{L^2_h(\Omega)}\|\nabla_h \varphi\|_{L^2_h(\Omega)} \leq CRh^3.
	\end{align*}
	Lastly we have to deal with the rotational term. Using the momentum balance law, $u_{0,h}$, $u_{1,h}\in\mathcal{B}$ and the structure of $g$, we obtain with $q^h := h^2 f^h$
	\begin{align*}
	\int_0^t \int_\Omega u_h\cdot x^\perp dxd\tau =& t \int_\Omega u_{0,h} \cdot x^\perp dx + \frac{1}{2} t^2\int_\Omega u_{1,h}\cdot x^\perp dx + \int_0^t (t-s) \int_\Omega q^h \cdot x^\perp dxds \\
	& +\frac{1}{h} \int_0^t \int_0^\tau (\tau-s) \int_\Omega q^h \cdot u_h^\perp - u_h^\perp \cdot \partial_t^2 u_h dxdsd\tau\\
	= & \frac{1}{h} \int_0^t \int_0^\tau (\tau-s) \int_\Omega q^h \cdot u_h^\perp -  u_h^\perp \cdot \partial_t^2 u_h dxdsd\tau.
	\end{align*}
	Hence it follows
	\begin{align*}
	\bigg\|\int_0^t \frac{1}{h} \int_\Omega u_h\cdot x^\perp dxd\tau \bigg\|_{C^0([0,T])} &\leq C \bigg(\bigg\|\frac{1}{h^2}\int_\Omega q^h \cdot u_h^\perp dx\bigg\|_{C^0([0,T])} + \bigg\|\frac{1}{h^2}\int_\Omega \partial_t^2 u_h \cdot u_h^\perp dx\bigg\|_{C^0([0,T])} \bigg)\\
	& \leq Ch^{3}
	\end{align*}
	as due to \eqref{GlobalMainTheoremInequality}
	\begin{equation*}
	\bigg\|\frac{1}{h}\e(\partial_t^\delta u_h)\bigg\|_{L^\infty(0,T;L^2)} + \bigg\|\frac{1}{h}\int_\Omega \partial_t^\delta u_h\cdot x^\perp dx\bigg\|_{L^\infty(0,T)} \leq Ch^{2}
	\end{equation*}
	for $\delta = 0,2$. Thus with \eqref{InequalitieLinearisedWeakForm} it follows
	\begin{equation*}
	\bigg\|\bigg((u_h - \tilde{u}_h), \frac{1}{h} \int_0^t \e_h\big(u_h(\tau) - \tilde{u}_h(\tau)\big) d\tau\bigg)\bigg\|_{C^0([0,T];L^2)} \leq Ch^{3}.\qedhere
	\end{equation*}
\end{proof}
\appendix
\section{Large Times Existence for the Non-linear Problem}
\label{Appendix::ExistenceOfClassicalSolutionsForFixedH}
The existence of solutions follows from
\begin{theorem}
	Let $\theta \geq 1$, $0< T< \infty$, $f_h\in W^3_1(0,T; L^2(\Omega)) \cap W^1_1(0,T; H^2_{per}(\Omega))$, $h\in (0,1]$ and $u_{0,h}\in H^4_{per}(\Omega)$, $u_{1,h}\in H^3_{per}(\Omega)$ such that
	\begin{align*}
	D\tilde{W}(\nabla_h u_{0,h})\nu|_{(0,L)\times\partial S} = D^2\tilde{W}(\nabla_h u_{0,h})[\nabla_h u_{1,h}]\nu|_{(0,L)\times\partial S} = 0,\\
	(D^2\tilde{W}(\nabla_h u_{0,h})[\nabla_h u_{2,h}]+ D^3\tilde{W}(\nabla_h u_{0,h})[\nabla_h u_{1,h}, \nabla_h u_{1,h}])\nu|_{(0,L)\times\partial S} = 0,
	\end{align*}
	where 
	\begin{align*}
	u_{2,h} &= h^{1+\theta} f_h|_{t=0} + \frac{1}{h^2}\divh(D\tilde{W}(\nabla_h u_{0,h}))\\
	u_{3,h} & = h^{1+\theta} \partial_t f_h|_{t=0} + \frac{1}{h^2} \divh(D^2\tilde{W}(\nabla_h u_{0,h})\nabla_h u_{1,h})\\
	u_{4,h} & = h^{1+\theta} \partial_t^2 f_h|_{t=0} + \frac{1}{h^2} \divh(D^2\tilde{W}(\nabla_h u_{0,h})\nabla_h u_{2,h})\\
	& \qquad + \frac{1}{h^2} \divh(D^3\tilde{W}(\nabla_h u_{0,h})[\nabla_h u_{1,h}, \nabla_h u_{1,h}]).
	\end{align*}
	Moreover we assume for the initial data
	\begin{gather}
		\Big\|\frac{1}{h}\e_h(u_{0,h})\Big\|_{H^2} + \max_{k=0,1,2} \Big\|\Big(\frac{1}{h}\e_h(u_{1+k, h}), \partial_{x_1}\frac{1}{h}\e_h(u_{k,h}), u_{2+k,h}\Big)\Big\|_{H^{2-k}} \leq Mh^{1+\theta}, \label{AssumptionOnInitialDataU1}\\
		\Big\|\nabla_h^2 u_{0,h}\Big\|_{H^1} + \max_{k=0,1} \Big\|\Big(\nabla_h^2 u_{1+k, h}, \partial_{x_1}\nabla_h^2 u_{k,h}\Big)\Big\|_{H^{1-k}} \leq Mh^{1+\theta}, \label{AssumptionOnInitialDataU2}\\
		\max_{k=0,1,2,3} \bigg|\frac{1}{h} \int_\Omega u_{k,h}\cdot x^\perp dx\bigg| \leq Mh^{1+\theta}\label{AssumptionOnInitialDataU3}
	\end{gather}
	and for the right hand side
	\begin{gather}
		\max_{|\alpha|\leq 1} \bigg(\|\partial_{(t,x_1)}^\alpha f_h\|_{W^2_1(L^2)} + \|\partial_{(t,x_1)}^\alpha f_h\|_{W^1_\infty(L^2)\cap W^1_1(H^{0,1})} + \|\partial_{(t,x_1)}^\alpha f_h\|_{L^\infty(H^1)}\bigg) \leq M, \label{AssumptionsOnNonlinearf}\\
		\max_{\sigma =0,1,2} \bigg\| \frac{1}{h} \int_\Omega \partial_t^\sigma f_h\cdot x^\perp dx\bigg\|_{C^0([0,T])} \leq M \label{AssumptionsOnRotationOfNonlinearf}
	\end{gather}
	uniformly in $0<h\leq 1$.
	Then there exists $h_0\in (0,1]$ and $C>0$ depending only on $M$ and $T$ such that for every $h\in (0,h_0]$ there is a unique solution $u_h\in \bigcap_{k=0}^4 C^k([0,T]; H^{4-k}_{per}(\Omega))$ of \eqref{NLS1}--\eqref{NLS4} satisfying 
	\begin{align}
	&\max_{\substack{|\alpha| \leq 1, |\beta|\leq 2, |\gamma|\leq 1\\\sigma=0,1,2}} \bigg(\Big\|\Big(\partial_t^2\partial_t^\sigma u_h, \nabla_{x,t}^\beta\frac{1}{h}\e_h(\partial_{(t,x_1)}^\alpha u_h), \nabla_{x,t}^\gamma\nabla_h^2 \partial_{(t,x_1)}^\alpha u_h\Big)\Big\|_{C^0([0,T], L^2)}\notag \\
	&\hspace{3cm}+ \bigg\|\frac{1}{h} \int_\Omega \partial_{(t,x_1)}^{\alpha + \beta} u_h \cdot x^\perp dx\bigg\|_{C^0([0,T])}\bigg) \leq C h^{1+\theta} \label{GlobalMainTheoremInequality}
	\end{align}
	uniformly in $0<h\leq h_0$.
	\label{MainTheorem}
\end{theorem}
\begin{proof}
	A proof can be found in \cite[Theorem 5.1.1]{AmeismeierDiss} or \cite{AbelsAmeismeier1}.
\end{proof}
The linearised system for \eqref{NLS1}--\eqref{NLS4} is given by 
\begin{align}
\partial_t^2 w - \frac{1}{h^2}\operatorname{div}_h (D^2\tilde{W}(\nabla_h u_h) 
\nabla_h w) &= f\quad\text{in }\Omega\times [0, T) \label{linearEQ_1}\\
D^2\tilde{W}(\nabla_h u)[\nabla_h w]\nu &=0 \quad\text{on } (0,L)\times\partial S\times [0, T)\label{linearEQ_2}\\
w \text{ is $L$-periodic}& \text{ in $x_1$ coordinate}\label{linearEQ_3}\\
(w, \partial_t w)|_{t=0} &= (w_0, w_1).\label{linearEQ_4}
\end{align}
We want to show $h$-independent estimates for solutions of the linearised system. For this we assume that $u_h$ satisfies for $0 < h \leq 1$
\begin{align}
\sup_{|\alpha|\leq 1, k=0,1,2} \bigg(\Big\|\Big(\nabla_{x,t}^k\frac{1}{h} &\e_h(\partial_{(t,x_1)}^\alpha u_h), \nabla_{x,t}^k \nabla_h \partial_{(t,x_1)}^\alpha u_h\Big)\Big\|_{C^0([0,T];L^2(\Omega))} \notag\\
& + \bigg\|\frac{1}{h}\int_\Omega \partial_{(t,x_1)}^{\alpha+\beta} u_h\cdot x^\perp dx\bigg\|_{C^0([0,T])} \bigg) \leq Rh,
\label{AssumptionsOnU1}
\end{align}
where $R\in (0,R_0]$, with $R_0$ chosen later appropriately small.
\begin{lem}
	Assume that \eqref{AssumptionsOnU1} holds, $t\in [0, T]$ and $0< R\leq R_0$. Then
	\begin{equation}
	\left| \frac{1}{h^2} \left(\partial_{z}^\beta D^2\tilde{W}(\nabla_h u_h 
	(t)) \nabla_h w, \nabla_h v)\right)_{L^2(\Omega)}\right| \leq CR \|\nabla_h w\|_{H^{|\beta|-1}_h(\Omega)} \|\nabla_h v\|_{L^2_h(\Omega)}
	\label{equationAbschätzungenAbleitungD2W}
	\end{equation}
	for $1 \leq |\beta| \leq 3$.
	\label{AbschätzungenAbleitungD2W}
\end{lem}
\begin{proof}
	For a proof see \cite[Lemma 5.2.2]{AmeismeierDiss}.
\end{proof}
For the higher regularity estimates we need the following result.
\begin{theorem}
	Assume $u_h$ satisfies \eqref{AssumptionsOnU1}. Then there exist $C>0$ 
	and $R_0\in (0,1]$ such that if $\varphi\in H^{2+k}_{per}(\Omega)$ solves for some $g\in H^k_{per}(\Omega)$ and $g_N\in L^2(0,L;H^{k+\frac{1}{2}}(\partial S))\cap H^{k}(0,L; H^{\frac{1}{2}}(\partial S))$
	\begin{equation}
	\left\{
	\begin{aligned}
	- \frac{1}{h^2} \operatorname{div}_h (D^2\tilde{W}(\nabla_h u_h)\nabla_h\varphi) &= g &&\quad\text{in } \Omega, \\
	D^2\tilde{W}(\nabla_h u_h)[\nabla_h \varphi]\nu\Big|_{(0,L)\times \partial S} &= g_N &&\quad\text{on } \partial\Omega,
	\end{aligned}
	\right.
	\label{LocalStationaryNonZeroNeumannBoundaryEquationHigherRegularity}
	\end{equation}
	then
	\begin{align}\nonumber
	\bigg\| \bigg(\nabla\frac{1}{h}\e_h(\varphi), \nabla_h^2
	\varphi\bigg)\bigg\|_{H^k(\Omega)} &\leq C\bigg(h^2\|g\|_{L^2(\Omega)} + \bigg\|\frac{1}{h} g_N\bigg\|_{L^2(0,L;H^{k+\frac{1}{2}}(\partial S))\cap H^{k}(0,L; H^{\frac{1}{2}}(\partial S))}\\
	& \qquad + \bigg\|\frac{1}{h}\e_h(\varphi)\bigg\|_{H^{0,k+1}(\Omega)} + R
	\bigg|\frac{1}{h} \int_\Omega \varphi\cdot x^\perp dx\bigg|\bigg).
	\label{InequalityHigherRegularity}
	\end{align}
	\label{TheoremHigherRegularity}
\end{theorem}
\begin{proof}
 See \cite[Theorem~3.5]{AbelsAmeismeier1}.
\end{proof}
In order to bound differences between the approximation $\tilde{u}_h$ and the analytic solution $u_h$ we consider the following weak form of the linearised system \eqref{linearEQ_1}--\eqref{linearEQ_4}:		
\begin{align}
-\big(\partial_t w, \partial_t\varphi\big)_{L^2(Q_T)} + \frac{1}{h^2} \big(D^2\tilde{W}(\nabla_h u_h) \nabla_h w, \nabla_h \varphi\big)_{L^2(Q_T)} &= \big(f_1, \nabla_h \varphi\big)_{L^2(Q_T)} + \big(f_2, \varphi\big)_{L^2(Q_T)}\notag\\
+ \langle w_1, \varphi|_{t=0} \rangle_{X'_h, X_h} + \frac{1}{h^2} (\tr{\partial\Omega}&(a_N), \tr{\partial\Omega}(\varphi))_{L^2(0,T;L^2(\partial\Omega))} \label{GlobalWeakLinearisedSystem}\\
w \text{ is } L\text{-periodic in }& x_1 \text{ direction}\notag\\
w|_{t=0} &= w_0\notag
\end{align}
for all $\varphi\in C^1([0,T]; H^1_{per, (0)}(\Omega;\R^3))$ with $\varphi|_{t=T} = 0$. Here we denote $Q_T := \Omega\times (0,T)$ and 
\begin{equation*}
X_h := H^1_{per, (0)}(\Omega;\R^3) := H^1_{per}(\Omega;\R^3) \cap \bigg\{u\in L^1(\Omega;\R^3)\;:\; \int_\Omega u(x) dx=0 \bigg\}
\end{equation*}
equipped with the $h$ dependent norm
\begin{equation*}
\|u\|_{X_h} := \|\nabla_h u\|_{L^2_h(\Omega)}.
\end{equation*}		
\begin{lem}
	Assume that $u_h$ satisfies \eqref{AssumptionsOnU1} with $R\in (0,R_0]$ and $h\in (0,1]$. Let $R_0$ be sufficiently small and $w\in C^0([0, T]; X_h)\cap C^1([0,T];L^2(\Omega;\R^3))$ be a solution of \eqref{GlobalWeakLinearisedSystem} for $f_1\in L^1(0,T;L^2(\Omega, \R^{3\times 3}))$, $f_2 \in L^1(0,T;L^2(\Omega;\R^3))$, $a_N\in L^1(0,T; H^1(\Omega;\R^3))$ $w_0\in L^2(\Omega;\R^3)$ and $w_1\in X'_h$. Then there are $C_0$, $C>0$ independent of $w$ and $T$ such that
	\begin{align}
	\bigg\|\bigg(w, \frac{1}{h} \e_h(u)\bigg)\bigg\|_{C^0([0,T];L^2)} &\leq C_0 e^{CRT} \bigg(\|f_1\|_{L^1(0,T; (L^2_h)')} + \|f_2\|_{L^1(0,T;L^2)} + \|w_0\|_{L^2(\Omega)} + \|w_1\|_{X'_h} \nonumber\\
	&+ \frac{1}{h^2} \|a_N\|_{L^1(0,T;H^1)} + (1+T)\bigg\|\frac{1}{h} \int_\Omega u\cdot x^\perp dx\bigg\|_{C^0([0,T])}\bigg),\label{InequalitieLinearisedWeakForm}
	\end{align}
	where $u(t) := \int_0^t w(\tau) d\tau$ and $(L^2_h)'$ is an abbreviation for $(L^2_h(\Omega;\Rtimes))'$.
\end{lem}
\begin{proof}
	Let $0 \leq T' \leq T$ and define $\tilde{u}_{T'}(t) = -\int_t^{T'} w(\tau) d\tau$. We use, after smooth approximation, $\varphi = \tilde{u}_{T'} \chi_{[0,T']}$. Then it follows
	\begin{align*}
	\frac{1}{2} \|w(&T')\|^2_{L^2} + \frac{1}{2 h^2} \Big(D^2\tilde{W}(\nabla_h u_h|_{t=0}) \nabla_h \tilde{u}_{T'}(0), \nabla_h \tilde{u}_{T'}(0)\Big)_{L^2(\Omega)}\\
	&= -\frac{1}{2h^2} \Big(\partial_t D^2\tilde{W}(\nabla_h u_h) \nabla_h \tilde{u}_{T'}, \nabla_h \tilde{u}_{T'}\Big)_{L^2(Q_{T'})} - (f_1, \nabla_h \tilde{u}_{T'})_{L^2(Q_{T'})} - (f_2, \tilde{u}_{T'})_{L^2(Q_{T'})}\\
	& \quad + \langle w_1, \tilde{u}_{T'}(0) \rangle_{X'_h, X_h} - \frac{1}{h^2} (a_N, \operatorname{tr}_{\partial\Omega}(\tilde{u}_{T'}))_{L^2(0,T'; L^2(\partial\Omega))} + \frac{1}{2} \|w(0)\|^2_{L^2}.
	\end{align*}
	Using 
	\begin{align*}
	\frac{1}{2 h^2} \Big(D^2\tilde{W}(\nabla_h u_h|_{t=0}) \nabla_h \tilde{u}_{T'}(0), &\nabla_h \tilde{u}_{T'}(0)\Big)_{L^2(\Omega)}\\
	& \geq \frac{c_0}{2} \bigg\|\frac{1}{h} \e_h(\tilde{u}_{T'}(0))\bigg\|^2_{L^2(\Omega)} - CR \bigg|\frac{1}{h} \int_\Omega \tilde{u}_{T'}(0) \cdot x^\perp dx\bigg|^2
	\end{align*}
	it follows with $\tilde{u}_{T'}(0) = -u(T')$
	\begin{align*}
	\|w&(T')\|^2_{L^2} + \bigg\|\frac{1}{h} \e_h(u(T'))\bigg\|^2_{L^2} \leq CR \int_0^{T'} \bigg\|\frac{1}{h}\e_h(\tilde{u}_{T'})\bigg\|^2_{L^2} + \bigg|\frac{1}{h}\int_\Omega \tilde{u}_{T'}\cdot x^\perp dx\bigg|^2 dt\\
	& + C\Big(\|f_1\|_{L^1(0,T; (L^2_h)')} + \|f_2\|_{L^1(0,T;L^2)} + \|w_1\|_{X'_h} + \frac{1}{h^2}\|a_N\|_{L^1(0,T;H^1)}\Big) \|\nabla_h \tilde{u}_{T'}\|_{C^0([0,T];L^2_h)}\\
	& + C\|w_0\|^2_{L^2} + CR\bigg|\frac{1}{h}\int_\Omega \tilde{u}_{T'}(0) \cdot x^\perp dx\bigg|^2
	\end{align*}
	where we used Lemma \ref{AbschätzungenAbleitungD2W} and Korn's inequality, as well as the subsequent inequalities
	\begin{align*}
	& |\langle w_1, \tilde{u}_{T'}(0) \rangle_{X'_h, X_h}| \leq \|w_1\|_{X'_h} \|\tilde{u}_{T'}(0)\|_{X_h} \leq \|w_1\|_{X'_h} \|\nabla_h \tilde{u}_{T'}\|_{C^0([0,T'];L^2_h)}\\
	& |(f_1, \nabla_h \tilde{u}_{T'})_{Q_{T'}}| \leq \int_0^{T'} \|f_1(t)\|_{(L^2_h)'} \|\nabla_h \tilde{u}_{T'}\|_{L^2} dt \leq \|f_1(t)\|_{L^1(0,T;(L^2_h)')} \|\nabla_h \tilde{u}_{T'}\|_{C^0([0,T'];L^2_h)}.
	\end{align*}
	Now we can use $\tilde{u}_{T'}(0) = -u(T')$ and $\tilde{u}_{T'}(t) = -u(T') + u(t)$ to deduce
	\begin{align*}
	\bigg|\int_0^{T'} \bigg\|\frac{1}{h}\e_h(\tilde{u}_{T'}(t))\bigg\|^2_{L^2} dt\bigg| &\leq \bigg\|\frac{1}{h}\e_h(u)\bigg\|^2_{L^2(Q_T)} + T' \bigg\|\frac{1}{h} \e_h(u(T'))\bigg\|^2_{L^2(\Omega)},\\
	\|\nabla_h \tilde{u}_{T'} \|_{C^0([0,T']; L^2_h(\Omega))} &\leq C \bigg\|\frac{1}{h} \e_h(u) \bigg\|_{C^0([0,T']; L^2(\Omega))} + C\bigg\|\frac{1}{h}\int_\Omega u \cdot x^\perp dx\bigg\|_{C^0([0,T'])},\\
	\bigg|\int_0^{T'}\int_\Omega \tilde{u}_{T'} \cdot x^\perp dx dt \bigg|&\leq T'\bigg|\int_\Omega u(T')\cdot x^\perp dx\bigg| + T' \bigg\|\int_\Omega u\cdot x^\perp dx\bigg\|_{C^0([0,T'])}.
	\end{align*}
	Using the later inequalities and applying the supremum over $T'\in [0, \bar{T}]$ such that $R\bar{T} \leq \kappa$, $\kappa\in (0,1]$ it follows
	\begin{align*}
	\|w&\|^2_{C^0([0,\bar{T}], L^2)} + \bigg\|\frac{1}{h} \e_h(u)\bigg\|^2_{C^0([0,\bar{T}];L^2)} \leq CR \bigg\|\frac{1}{h} \e_h(u)\bigg\|^2_{L^2(Q_{\bar{T}})} + C\kappa \bigg\|\frac{1}{h} \e_h(u)\bigg\|^2_{C^0([0,\bar{T}];L^2)}\\
	& \quad+ C\Big(\|f_1\|_{L^1(0,T; (L^2_h)')} + \|f_2\|_{L^1(0,T;L^2)} + \|w_1\|_{X'_h} + \|w_1\|_{X'_h} + \frac{1}{h^2}\|a_N\|_{L^1(0,T;H^1)}\Big)\\
	&\qquad\quad  \times \bigg(\bigg\|\frac{1}{h} \e_h(u) \bigg\|_{C^0([0,T']; L^2)} + C\bigg\|\frac{1}{h}\int_\Omega u \cdot x^\perp dx\bigg\|_{C^0([0,T'])}\bigg)\\
	&\quad + C\|w_0\|^2_{L^2} + CR(1+\bar{T})\bigg\|\frac{1}{h}\int_\Omega u\cdot x^\perp dx\bigg\|^2_{C^0([0,\bar{T}])}
	\end{align*}
	Hence, with Young's inequality and $\kappa$, thus $\bar{T}$, small enough, we can conclude with an absorption argument that
	\begin{align*}
	\|w&\|^2_{C^0([0,\bar{T}], L^2)} + \bigg\|\frac{1}{h} \e_h(u)\bigg\|^2_{C^0([0,\bar{T}];L^2)} \leq CR \bigg\|\frac{1}{h} \e_h(u)\bigg\|^2_{L^2(Q_{\bar{T}})} + C_0 \bigg(\|f_1\|^2_{L^1(0,T; (L^2_h)')} \\
	&\quad + \|f_2\|^2_{L^1(0,T;L^2)} + \|w_1\|^2_{X'_h} + \frac{1}{h^4}\|a_N\|^2_{L^1(0,T;H^1)} + (1+T)\bigg\|\frac{1}{h}\int_\Omega u\cdot x^\perp dx\bigg\|^2_{C^0([0,T])}\bigg).
	\end{align*}
	Applying now the Lemma of Gronwall we obtain \eqref{InequalitieLinearisedWeakForm} for all $0 < T < \infty$ such that $RT \leq \kappa$ holds.
	
	For an arbitrary $0< T <\infty$, we choose $0=T_0  < T_1 < \ldots < T_{N-1} < T_N = T$ such that $\frac{1}{2}\kappa \leq R(T_{j+1} - T_j) \leq \kappa$ for $j=0,\ldots N-1$. Then we use $\varphi = \tilde{u}_{T_{j+1}} \chi_{[T_{j},T_{j+1}]}$ and obtain via analogous arguments as above, because of $R(T_{j+1} - T_j) \leq \kappa$, 
	\begin{align*}
	\bigg\|\bigg(w, \frac{1}{h} \e_h(u)\bigg)\bigg\|_{C^0([T_j,T_{j+1}];L^2)} \leq C_0 &e^{CR(T_{j+1} - T_j)} \bigg(\bigg\|\bigg(w(T_j), \frac{1}{h}\e_h(u(T_j))\bigg)\bigg\|_{L^2} \\
	&  + \|f_1\|_{L^1(0,T; (L^2_h)')} + \|f_2\|_{L^1(0,T;L^2)} + \|w_1\|_{X'_h}\\
	&  + \frac{1}{h} \|a_N\|_{L^1(0,T;H^1)} + (1+T)\bigg\|\frac{1}{h} \int_\Omega u\cdot x^\perp dx\bigg\|_{C^0([0,T])} \bigg).
	\end{align*}
	Hence an iterative application leads to 
	\begin{align*}
	\bigg\|\bigg(w, \frac{1}{h} \e_h(u)\bigg)\bigg\|_{C^0([0,T];L^2)} &\leq (C_0)^N e^{CRT} \bigg(\|f_1\|_{L^1(0,T; X'_h)} + \|f_2\|_{L^1(0,T;L^2)} + \|w_0\|_{L^2(\Omega)}  \\
	& \; + \|w_1\|_{X'_h} + \frac{1}{h} \|a_N\|_{L^1(0,T;H^1)} + (1+T)\bigg\|\frac{1}{h} \int_\Omega u\cdot x^\perp dx\bigg\|_{C^0([0,T])} \bigg).
	\end{align*}
	Finally due to $\frac{1}{2}\kappa \leq R(T_{j+1} - T_j)$, we obtain $N\leq 2\kappa^{-1}RT$ and thus
	\begin{equation*}
	(C_0)^N = \exp(N\ln C_0) \leq \exp( 2\kappa^{-1}RT \ln C_0) \leq \exp(C'_0 RT).
	\end{equation*}
	Hence \eqref{InequalitieLinearisedWeakForm} holds for some $C_0$, $C >0$ independent of $R\in (0,R_0]$, $h\in (0, 1]$ and $0 < T <\infty$.
\end{proof}

\bibliographystyle{abbrv}
\bibliography{bib.bib}

\end{document}